\documentclass{daj}

\usepackage[nobysame, initials]{amsrefs}
\usepackage{amsmath, amsthm, amssymb}
\allowdisplaybreaks

\usepackage{enumitem}
\def\itm#1{\rm ({#1})} 
\def\itmit#1{\itm{\it #1\;\!}} 
\def\rom{\itmit{\roman{*}}} 

\def\l{\ell}
\def\eps{\varepsilon}
\def\phi{\varphi}
\def\le{\leqslant}
\def\ge{\geqslant}
\def\CC{\mathbb{C}}
\def\FF{\mathbb{F}}
\def\PP{\mathbb{P}}
\def\RR{\mathbb{R}}
\def\ZZ{\mathbb{Z}}
\newcommand{\setbuilder}[2]{\left\{#1:#2\right\}}
\newcommand{\tri}{\mathbin{\triangle}}
\DeclareMathOperator{\rank}{rank}
\DeclareMathOperator{\Sec}{Sec}
\DeclareMathOperator{\Tan}{Tan}

\newtheorem{theorem}{Theorem}[section]
\newtheorem{lemma}[theorem]{Lemma}
\newtheorem{corollary}[theorem]{Corollary}
\newtheorem{prop}[theorem]{Proposition}
\theoremstyle{definition}
\newtheorem*{definition}{Definition}

\dajAUTHORdetails{%
  title = {On sets defining few ordinary hyperplanes}, %% please capitalize all significant words
  author = {Aaron Lin and Konrad Swanepoel},
  plaintextauthor = {Aaron Lin, Konrad Swanepoel},
}   %%% END \dajAUTHORdetails

\dajEDITORdetails{%
   year={2020},
   number={4},
   received={26 April 2019},   % received date: example: 7 January 2017
   revised={17 January 2020},    % Optional revised date (you may comment it out)
   published={24 April 2020},  % published date
   doi={10.19086/da.11949},       % XXX = number of paper, e.g. da006 for paper#6
}   %%% END \dajEDITORdetails

\begin{document}

\begin{frontmatter}[classification=text]
%% EDITOR: this will force the keywords to appear right after the Abstract.
%%   If the abstract is too long and would force the keywords off the
%%   front page, please comment out % [classification=text] above
%%   This way the keywords will be floated on the bottom of the first page
%%   even though the Abstract spills over to the next page.

%%% AUTHOR: Title goes here.  This line is optional.  You must use it
%%   if title has footnote attached or requires nontrivial typesetting,
%%   e.g., inclusion of linebreaks to force nice layout.
%\title{Short Proof of R\"odl's $n^{\log\log n}$ Bound\footnote{This is a footnote to the title}} %% please capitalize all significant words

%%% AUTHOR:
%%% List all authors. If you wish, place grant acknowledgements in \thanks.
%%% In brackets include a short tag for each author.
\author[AL]{Aaron Lin}
\author[KS]{Konrad Swanepoel}

%%% AUTHOR: Abstract goes here
\begin{abstract}
Let $P$ be a set of $n$ points in real projective $d$-space, not all contained in a hyperplane, 
such that any $d$ points span a hyperplane.
An ordinary hyperplane of $P$ is a hyperplane containing exactly $d$ points of $P$.
We show that if $d\ge 4$, the number of ordinary hyperplanes of $P$ is at least $\binom{n-1}{d-1} - O_d(n^{\lfloor(d-1)/2\rfloor})$ if $n$ is sufficiently large depending on $d$.
This bound is tight, and given $d$, we can calculate the exact minimum number for sufficiently large $n$.
This is a consequence of a structure theorem for sets with few ordinary hyperplanes: For any $d \ge 4$ and $K > 0$, if $n \ge C_d K^8$ for some constant $C_d > 0$ depending on $d$,
and $P$ spans at most $K\binom{n-1}{d-1}$ ordinary hyperplanes, then all but at most $O_d(K)$ points of $P$ lie on a hyperplane, an elliptic normal curve, or a rational acnodal curve.
We also find the maximum number of $(d+1)$-point hyperplanes, solving a $d$-dimensional analogue of the orchard problem.
Our proofs rely on Green and Tao's results on ordinary lines, our earlier work on the $3$-dimensional case, as well as results from classical algebraic geometry.
\end{abstract}
\end{frontmatter}

%%% AUTHOR: body of paper starts here
\section{Introduction}\label{sec:intro}

An \emph{ordinary line} of a set of points in the plane is a line passing through exactly two points of the set.
The classical Sylvester--Gallai theorem states that every finite non-collinear point set in the plane spans at least one ordinary line.
In fact, for sufficiently large $n$, an $n$-point non-collinear set in the plane spans at least $n/2$ ordinary lines, and this bound is tight if $n$ is even.
This was shown by Green and Tao \cite{GT13} via a structure theorem characterising all finite point sets with few ordinary lines.

It is then natural to consider higher dimensional analogues.
Motzkin \cite{M51} noted that there are finite non-coplanar point sets in $3$-space that span no plane containing exactly three points of the set.
He proposed considering instead hyperplanes $\Pi$ in $d$-space such that all but one point contained in $\Pi$ is contained in a $(d-2)$-dimensional flat of $\Pi$.
The existence of such hyperplanes was shown by Motzkin \cite{M51} for $3$-space and by Hansen \cite{H65} in higher dimensions.

Purdy and Smith \cite{PS10} considered instead finite non-coplanar point sets in $3$-space with no three points collinear, 
and provided a lower bound on the number of planes containing exactly three points of the set.
Referring to such a plane as an \emph{ordinary plane}, Ball \cite{B18} proved a $3$-dimensional analogue of Green and Tao's \cite{GT13} structure theorem, and found the exact minimum number of ordinary planes spanned by sufficiently large non-coplanar point sets in real projective $3$-space with no three points collinear.
Using an alternative method, we \cite{LS18} were able to prove a more detailed structure theorem but with a stronger condition; see Theorem~\ref{thm:plane} in Section~\ref{sec:proof}.

Ball and Monserrat \cite{BM17} made the following definition, generalising ordinary planes to higher dimensions.
\begin{definition}
An \emph{ordinary hyperplane} of a set of points in real projective $d$-space, where every $d$ points span a hyperplane, is a hyperplane passing through exactly $d$ points of the set.
\end{definition}
They \cite{BM17} also proved bounds on the minimum number of ordinary hyperplanes spanned by such sets (see also \cite{M15}).
Our first main result is a structure theorem for sets with few ordinary hyperplanes.
The elliptic normal curves and rational acnodal curves mentioned in the theorem and their group structure will be described in Section~\ref{sec:curves}.
Our methods extend those in our earlier paper \cite{LS18}, and we detail them in Section~\ref{sec:tools}.

\begin{theorem}\label{thm:main1}
Let $d \ge 4$, $K > 0$, and suppose $n \ge C\max\{(dK)^8, d^3 2^dK\}$ for some sufficiently large absolute constant $C > 0$.
Let $P$ be a set of $n$ points in $\RR\PP^d$ where every $d$ points span a hyperplane.
If $P$ spans at most $K\binom{n-1}{d-1}$ ordinary hyperplanes, then $P$ differs in at most $O(d2^dK)$ points from a configuration of one of the following types:
\begin{enumerate}[label=\rom]
\item A subset of a hyperplane;\label{case:hyperplane}
\item A coset $H \oplus x$ of a subgroup $H$ of an elliptic normal curve or the smooth points of a rational acnodal curve of degree $d+1$, for some $x$ such that $(d+1)x \in H$.\label{case:curve}
\end{enumerate}
\end{theorem}

It is easy to show that conversely, a set of $n$ points where every $d$ span a hyperplane and differing from \ref{case:hyperplane} or \ref{case:curve} by $O(K)$ points, spans
$O(K\binom{n-1}{d-1})$ ordinary hyperplanes.
By \cite{BM17}*{Theorem 2.4}, if a set of $n$ points where every $d$ points span a hyperplane itself spans $K\binom{n-1}{d-1}$ ordinary hyperplanes, and is not contained in a hyperplane, then $K = \Omega(1/d)$.
Theorem~\ref{thm:main2} below implies that $K\ge1$ for sufficiently large $n$ depending on $d$.

For a similar structure theorem in dimension $4$ but with $K = o(n^{1/7})$, see Ball and Jimenez \cite{BJ18}, who show that $P$ lies on the intersection of five quadrics.
Theorem~\ref{thm:main1} proves \cite{BJ18}*{Conjecture~12}, 
noting that elliptic normal curves and rational acnodal curves lie on $\binom{d}{2} - 1$ linearly independent quadrics \citelist{\cite{Klein}*{p.~365}\cite{Fi08}*{Proposition~5.3}}.
We also mention that Monserrat \cite{M15}*{Theorem 2.10} proved a structure theorem stating that almost all points of the set lie on the intersection of $d-1$ hypersurfaces of degree at most~$3$.

Our second main result is a tight bound on the minimum number of ordinary hyperplanes, proving \cite{BM17}*{Conjecture 3}.
Note that our result holds only for sufficiently large $n$; see \citelist{\cite{BM17}\cite{M15}\cite{J18}} for estimates when $d$ is small or $n$ is not much larger than $d$.

\begin{theorem}\label{thm:main2}
Let $d \ge 4$ and let $n\ge Cd^3 2^d$ for some sufficiently large absolute constant $C > 0$.
The minimum number of ordinary hyperplanes spanned by a set of $n$ points in $\RR\PP^d$, not contained in a hyperplane and where every $d$ points span a hyperplane, is
\[ \binom{n-1}{d-1} - O\left(d2^{-d/2}\binom{n}{\lfloor\frac{d-1}{2}\rfloor}\right).\]
This minimum is attained by a coset of a subgroup of an elliptic normal curve or the smooth points of a rational acnodal curve of degree $d+1$, and when $d+1$ and $n$ are coprime, by $n-1$ points in a hyperplane together with a point not in the hyperplane.
\end{theorem}

Green and Tao \cite{GT13} also used their structure theorem to solve the classical orchard problem of finding the maximum number of 
$3$-point lines 
spanned by a set of $n$ points in the plane, for $n$ sufficiently large.
We solved the $3$-dimensional analogue in \cite{LS18}.
Our third main result is the $d$-dimensional analogue.
We define a \emph{$(d+1)$-point hyperplane} to be a hyperplane through exactly $d+1$ points of a given set.

\begin{theorem}\label{thm:main3}
Let $d \ge 4$ and let $n\ge Cd^3 2^d$ for some sufficiently large absolute constant $C > 0$.
The maximum number of $(d+1)$-point hyperplanes spanned by a set of $n$ points in $\RR\PP^d$ where every $d$ points span a hyperplane is
\[ \frac{1}{d+1} \binom{n-1}{d} + O\left(2^{-d/2}\binom{n}{\lfloor\frac{d-1}{2}\rfloor}\right). \]
This maximum is attained by a coset of a subgroup of an elliptic normal curve or the smooth points of a rational acnodal curve of degree $d+1$.
\end{theorem}

While the bounds in Theorems~\ref{thm:main2} and~\ref{thm:main3} are asymptotic, we provide a recursive method (as part of our proofs) to calculate the exact extremal values for a given $d$ and $n$ sufficiently large in Section~\ref{sec:extremal}.
In principle, the exact values can be calculated for any given $d$ and turns out to be a quasi-polynomial in $n$ with a period of $d+1$.
We present the values for $d = 4, 5, 6$ at the end of Section~\ref{sec:extremal}.

\subsection*{Relation to previous work}
The main idea in our proof of Theorem~\ref{thm:main1} 
is to induct on the dimension $d$, with the base case $d=3$ being our earlier structure theorem for sets defining few ordinary planes \cite{LS18},
which in turn is based on Green and Tao's Intermediate Structure Theorem for sets defining few ordinary lines \cite{GT13}*{Proposition~5.3}.

Roughly, the structure theorem in $3$-space states that if a finite set of points is in general position (no three points collinear) and spans few ordinary planes, then most of the points must lie on a plane, two disjoint conics, or an elliptic or acnodal space quartic curve.
In fact, we can define a group structure on these curves encoding when four points are coplanar, in which case our point set must be very close to a coset of the curve. 
(See Theorem~\ref{thm:plane} 
for a more precise statement.)

As originally observed by Ball \cite{B18} in $3$-space, the general position condition allows the use of projection to leverage Green and Tao's Intermediate Structure Theorem \cite{GT13}*{Proposition~5.3}.
This avoids having to apply their Full Structure Theorem \cite{GT13}*{Theorem~1.5}, which has a much worse lower bound on $n$, as it avoids the technical Section~6 of \cite{GT13}, dealing with the case in the plane when most of the points lie on a large, though bounded, number of lines.
On the other hand, to get to the precise coset structure, we used additive-combinatorial results from \cite{GT13}*{Section~7}, specifically \cite{GT13}*{Propositions~A.5, Lemmas 7.2, 7.4, 7.7, and Corollary 7.6}.
In this paper, the only result of Green and Tao \cite{GT13} we explicitly use is \cite{GT13}*{Proposition~A.5}, which we extend in Proposition~\ref{cor:a5}, while all other results are subsumed in the structure theorem in $3$-space.
In dimensions $d>3$, the general position condition also allows the use of projections from a point to a hyperplane (see also Ball and Monserrat \cite{BM17}).
In Section~\ref{ssec:projections} we detail various technical results about the behaviour of curves under such projections, which are extensions of $3$-dimensional results in \cite{LS18}.

While the group structure on elliptic or singular space quartic curves are well studied (see for instance \cite{Muntingh}), we could not find references to the group structure on singular rational curves in higher dimensions.
This is our main focus in Section~\ref{sec:curves}, which in a way extends \cite{LS18}*{Section 3}.
In particular, we look at Sylvester's theorem on when a binary form can be written as a sum of perfect powers, which has its roots in classical invariant theory.
In extending the results of \cite{LS18}*{Section 3}, we have to consider how to generalise the catalecticant (of a binary quartic form), which leads us to the secant variety of the rational normal curve as a determinantal variety.

Green and Tao's Intermediate Structure Theorem in $2$-space has a slightly different flavour to their Full Structure Theorem, the structure theorem in $3$-space, and Theorem~\ref{thm:main1}.
However, this is not the only reason why we start our induction at $d=3$.
A more substantial reason is that there are no smooth rational cubic curves in $2$-space; as is well known, all rational planar cubic curves are singular.
Thus, both smooth and singular rational quartics in $3$-space project onto rational cubics, and we need some way to tell them apart.
In higher dimensions, we have Lemma~\ref{lemma:singular_projection}
to help us, but since this is false when $d=3$, the induction from the plane to $3$-space \cite{LS18} is more technical.
This is despite the superficial similarity between the $2$- and $3$-dimensional situations where there are two almost-extremal cases while there is essentially only one case when $d>3$.

Proving Theorem~\ref{thm:main1}, which covers the $d > 3$ cases, is thus in some sense less complicated, since not only are we leveraging a more detailed structure theorem (Theorems~\ref{thm:main1} 
and~\ref{thm:plane}
as opposed to \cite{GT13}*{Proposition~5.3}), we also lose a case.
However, there are complications that arise in how to generalise and extend results from $2$- and $3$-space to higher dimensions.

\section{Notation and tools}\label{sec:tools}

By $A = O(B)$, we mean there exists an absolute constant $C > 0$ such that $0\le A \le CB$.
Thus, $A=-O(B)$ means there exists an absolute constant $C>0$ such that $-CB\le A\le 0$.
We also write $A=\Omega(B)$ for $B=O(A)$.
None of the $O(\cdot)$ and $\Omega(\cdot)$ statements in this paper have implicit dependence on the dimension~$d$.

We write $A\tri B$ for the symmetric difference of the sets $A$ and $B$. 

Let $\FF$ denote the field of real or complex numbers, let $\FF^* = \FF\setminus{\{0\}}$, and let $\FF\PP^d$ denote the $d$-dimensional projective space over $\FF$.
We denote the homogeneous coordinates of a point in $d$-dimensional projective space by a $(d+1)$-dimensional vector $[x_0, x_1, \dots, x_d]$.
We call a linear subspace of dimension $k$ in $\FF\PP^d$ a \emph{$k$-flat}; thus a point is a $0$-flat, a line is a $1$-flat, a plane is a $2$-flat, and a hyperplane is a $(d-1)$-flat.
We denote by $Z_{\FF}(f)$ the set of $\FF$-points of the algebraic hypersurface defined by the vanishing of a homogeneous polynomial $f\in\FF[x_0, x_1, \dots, x_d]$.
More generally, we consider a (closed, projective) \emph{variety} to be any intersection of algebraic hypersurfaces.
We say that a variety is pure-dimensional if each of its irreducible components has the same dimension.
We consider a \emph{curve} of degree $e$ in $\CC\PP^d$ to be a variety $\delta$ of pure dimension $1$ such that a generic hyperplane in $\CC\PP^d$ intersects $\delta$ in $e$ distinct points.
More generally, the degree of a variety $X \subset \CC\PP^d$ of dimension $r$ is
\[ \deg(X) := \max \setbuilder{|\Pi \cap X|}{\text{$\Pi$ is a $(d-r)$-flat such that $\Pi \cap X$ is finite}}. \]

We say that a curve is \emph{non-degenerate} if it is not contained in a hyperplane, and \emph{non-planar} if it is not contained in a $2$-flat.
We call a curve \emph{real} if each of its irreducible components contains infinitely many points of $\RR\PP^d$.
Whenever we consider a curve in $\RR\PP^d$, we implicitly assume that its Zariski closure is a real curve.

We denote the Zariski closure of a set $S \subseteq \CC\PP^d$ by $\overline{S}$.
We will use the \emph{secant variety $\Sec_{\CC}(\delta)$} of a curve $\delta$, which is the Zariski closure of the set of points in $\CC\PP^d$ that lie on a line through some two points of $\delta$. 

\subsection{B\'ezout's theorem}\label{ssec:bezout}

B\'ezout's theorem gives the degree of an intersection of varieties.
While it is often formulated as an equality, in this paper we only need the weaker form that ignores multiplicity and gives an upper bound.
The (set-theoretical) intersection $X\cap Y$ of two varieties is just the variety defined by $P_X\cup P_Y$, where $X$ and $Y$ are defined by the collections of homogeneous polynomials $P_X$ and $P_Y$ respectively.

\begin{theorem}[B\'ezout \cite{Fu84}*{Section~2.3}]\label{thm:bezout}
Let $X$ and $Y$ be varieties in $\CC\PP^d$ with no common irreducible component.
Then $\deg(X \cap Y) \le \deg(X) \deg(Y)$.
\end{theorem}

\subsection{Projections}\label{ssec:projections}

Given $p\in\FF\PP^d$, the \emph{projection from $p$}, $\pi_p\colon \FF\PP^d\setminus\{p\}\to \FF\PP^{d-1}$, is defined by identifying $\FF\PP^{d-1}$ with any hyperplane $\Pi$ of $\FF\PP^d$ not passing through $p$, and then letting $\pi_p(x)$ be the point where the line $px$ intersects $\Pi$ \cite{H92}*{Example~3.4}.
Equivalently, $\pi_p$ is induced by a surjective linear transformation $\FF^{d+1}\to\FF^d$ where the kernel is spanned by the vector~$p$.

As in our previous paper \cite{LS18}, we have to consider projections of curves where we do not have complete freedom in choosing a generic projection point $p$.

Let $\delta \subset \CC\PP^d$ be an irreducible non-planar curve of degree $e$, and let $p$ be a point in $\CC\PP^d$.
We call $\pi_p$ \emph{generically one-to-one on $\delta$} if there is a finite subset $S$ of $\delta$ such that $\pi_p$ restricted to $\delta\setminus S$ is one-to-one.
(This is equivalent to the birationality of $\pi_p$ restricted to $\delta\setminus\{p\}$ \cite{H92}*{p.~77}.)
If $\pi_p$ is generically one-to-one, the degree of the curve $\overline{\pi_p(\delta \setminus \{p\})}$ is $e-1$ if $p$ is a smooth point on $\delta$, and is $e$ if $p$ does not lie on~$\delta$;
if $\pi_p$ is not generically one-to-one, then the degree of $\overline{\pi_p(\delta \setminus \{p\})}$ is at most $(e-1)/2$ if $p$ lies on $\delta$, and is at most $e/2$ if $p$ does not lie on~$\delta$
\cite{H92}*{Example 18.16}, \cite{Kollar}*{Section 1.15}.

The following three lemmas on projections are proved in \cite{LS18} in the case $d=3$.
They all state that most projections behave well and can be considered to be quantitative versions of the trisecant lemma \cite{KKT08}.
The proofs of Lemmas~\ref{lem:projection2} and \ref{lem:projection3} are almost word-for-word the same as the proofs of the $3$-dimensional cases in \cite{LS18}.
All three lemmas can also be proved by induction on the dimension $d \ge 3$ from the $3$-dimensional case.
We illustrate this by proving Lemma~\ref{lem:projection1}.

\begin{lemma}\label{lem:projection1}
Let $\delta$ be an irreducible non-planar curve of degree $e$ in $\CC\PP^d$, $d\ge 3$.
Then there are at most $O(e^4)$ points $p$ on $\delta$ such that $\pi_p$ restricted to $\delta\setminus\{p\}$ is not generically one-to-one.
\end{lemma}

\begin{proof}
The case $d=3$ was shown in \cite{LS18}, based on the work of Furukawa \cite{Fu11}.
We next assume that $d \ge 4$ and that the lemma holds in dimension $d-1$.
Since $d>3$ and the dimension of $\Sec_\CC(\delta)$ is at most $3$ \cite{H92}*{Proposition~11.24}, there exists a point $p \in \CC\PP^d$ such that all lines through $p$ have intersection multiplicity at most $1$ with $\delta$.
It follows that the projection $\delta':=\overline{\pi_p(\delta)}$ of $\delta$ is a non-planar curve of degree $e$ in $\CC\PP^{d-1}$.
Consider any line $\ell$ not through $p$ that intersects $\delta$ in at least three distinct points $p_1,p_2,p_3$.
Then $\pi_p(\ell)$ is a line in $\CC\PP^{d-1}$ that intersects $\delta'$ in three points $\pi_p(p_1), \pi_p(p_2), \pi_p(p_3)$.
It follows that if $x\in\delta$ is a point such that for all but finitely many points $y\in\delta$, the line $xy$ intersects $\delta$ in a point other than $x$ or $y$, then $x':=\pi_p(x)$ is a point such that for all but finitely many points $y':=\pi_p(y)\in\delta'$, the line $x'y'$ intersects $\delta'$ in a third point.
That is, if $\pi_x$ restricted to $\delta$ is not generically one-to-one, then the projection map $\pi_{x'}$ in $\CC\PP^{d-1}$ restricted to $\delta'$ is not generically one-to-one.
By the induction hypothesis, there are at most $O(e^4)$ such points and we are done.
\end{proof}

\begin{lemma}\label{lem:projection2}
Let $\delta$ be an irreducible non-planar curve of degree $e$ in $\CC\PP^d$, $d\ge 3$.
Then there are at most $O(e^3)$ points $x \in \CC\PP^d \setminus \delta$ such that $\pi_x$ restricted to $\delta$ is not generically one-to-one.
\end{lemma}

\begin{lemma}\label{lem:projection3}
Let $\delta_1$ and $\delta_2$ be two irreducible non-planar curves in $\CC\PP^d$, $d\ge 3$, of degree $e_1$ and $e_2$ respectively. 
Then there are at most $O(e_1e_2)$ points $p$ on $\delta_1$ such that $\overline{\pi_p(\delta_1\setminus\{p\})}$ and $\overline{\pi_p(\delta_2\setminus\{p\})}$ coincide.
\end{lemma}

\section{Curves of degree \texorpdfstring{$d+1$}{d+1}}\label{sec:curves}

In this paper, irreducible non-degenerate curves of degree $d+1$ in $\CC\PP^d$ play a fundamental role.
Indeed, the elliptic normal curve and rational acnodal curve mentioned in Theorem~\ref{thm:main1} are both such curves.
In this section, we describe their properties that we need.
These properties are all classical, but we did not find a reference for the group structure on singular rational curves of degree $d+1$, and therefore consider this in detail.

It is well-known in the plane that there is a group structure on any smooth cubic curve or the set of smooth points of a singular cubic.
This group has the property that three points sum to the identity if and only if they are collinear.
Over the complex numbers, the group on a smooth cubic is isomorphic to the torus $(\RR/\ZZ)^2$, and the group on the smooth points of a singular cubic is isomorphic to $(\CC,+)$ or $(\CC^*,\cdot)$ depending on whether the singularity is a cusp or a node.
Over the real numbers, the group on a smooth cubic is isomorphic to $\RR/\ZZ$ or $\RR/\ZZ\times\ZZ_2$ depending on whether the real curve has one or two semi-algebraically connected components, and the group on the smooth points of a singular cubic is isomorphic to $(\RR,+)$, $(\RR,+)\times\ZZ_2$, or $\RR/\ZZ$ depending on whether the singularity is a cusp, a crunode, or an acnode. 
See for instance \cite{GT13} for a more detailed description.

In higher dimensions, it turns out that an irreducible non-degenerate curve of degree $d+1$ does not necessarily have a natural group structure, but if it has, the behaviour is similar to the planar case.
For instance, in $\CC\PP^3$, an irreducible non-degenerate quartic curve is either an elliptic quartic, with a group isomorphic to an elliptic curve such that four points on the curve are coplanar if and only if they sum to the identity, or a rational curve.
There are two types, or species, of rational quartics.
The rational quartic curves of the first species are intersections of two quadrics (as are elliptic quartics), they are always singular, and there is a group on the smooth points such that four points on the curve are coplanar if and only if they sum to the identity.
Those of the second species lie on a unique quadric, are smooth, and there is no natural group structure analogous to the other cases.
See \cite{LS18} for a more detailed account.
The picture is similar in higher dimensions.
\begin{definition}[Clifford \cite{Clifford}, Klein \cite{Klein}]
An \emph{elliptic normal curve} is an irreducible non-degenerate smooth curve of degree $d+1$ in $\CC\PP^d$ isomorphic to an elliptic curve in the plane.
\end{definition}
\begin{prop}[\cite{S09}*{Exercise 3.11 and Corollary 5.1.1}, \cite{S94}*{Corollary 2.3.1}]\label{prop:elliptic_group}
An elliptic normal curve $\delta$ in $\CC\PP^d$, $d\ge 2$, has a natural group structure such that $d+1$ points in $\delta$ lie on a hyperplane if and only if they sum to the identity.
This group is isomorphic to $(\RR/\ZZ)^2$.

If the curve is real, then the group is isomorphic to $\RR/\ZZ$ or $\RR/\ZZ\times\ZZ_2$ depending on whether the real curve has one or two semi-algebraically connected components.
\end{prop}

A similar result holds for singular rational curves of degree $d+1$.
Since we need to work with such curves and a description of their group structure is not easily found in the literature, we give a detailed discussion of their properties in the remainder of this section.

A \emph{rational curve} $\delta$ in $\FF\PP^d$ of degree $e$ is a curve that can be parametrised by the projective line,
\[ \delta\colon \FF\PP^1 \to \FF\PP^d, \quad [x,y] \mapsto [q_0(x,y),\dots,q_{d}(x,y)], \]
where each $q_i$ is a homogeneous polynomial of degree $e$ in the variables $x$ and $y$.
The following lemma is well known (see for example \cite{SR85}*{p.~38, Theorem~VIII}), and can be proved by induction from the planar case using projection.

\begin{prop}\label{prop:curves}
An irreducible non-degenerate curve of degree $d+1$ in $\CC\PP^d$, $d \ge 2$, is either an elliptic normal curve or rational.
\end{prop}

We next describe when an irreducible non-degenerate rational curve of degree $d+1$ in $\CC\PP^d$ has a natural group structure.
It turns out that this happens if and only if the curve is singular.

We write $\nu_{d+1}$ for the \emph{rational normal curve} in $\CC\PP^{d+1}$ \cite{H92}*{Example~1.14}, which we parametrise as
\[ \nu_{d+1}:[x,y] \mapsto [y^{d+1},-xy^d,x^2y^{d-1},\dotsc,(-x)^{d-1}y^2,(-x)^dy,(-x)^{d+1}]. \]
Any irreducible non-degenerate rational curve $\delta$ of degree $d+1$ in $\CC\PP^d$ is the projection of the rational normal curve, and we have
\[ \delta[x,y] = [y^{d+1},-xy^d,x^2y^{d-1},\dotsc,(-x)^{d-1}y^2,(-x)^dy,(-x)^{d+1}] A, \]
where $A$ is a $(d+2)\times(d+1)$ matrix of rank $d+1$ (since $\delta$ is non-degenerate) with entries derived from the coefficients of the polynomials $q_i$ of degree $d+1$ in the parametrisation of the curve (with suitable alternating signs).
Thus $\delta \subset \CC\PP^d$ is the image of $\nu_{d+1}$ under the projection map $\pi_p$ defined by $A$.
In particular, the point of projection $p=[p_0,p_1,\dots,p_{d+1}]\in\CC\PP^{d+1}$ is the ($1$-dimensional) kernel of $A$.
If we project $\nu_{d+1}$ from a point $p\in \nu_{d+1}$, then we obtain a rational normal curve in $\CC\PP^d$.
However, since $\delta$ is of degree $d+1$, necessarily $p\notin \nu_{d+1}$.
Conversely, it can easily be checked that for any $p\notin \nu_{d+1}$, the projection of $\nu_{d+1}$ from $p$ is a rational curve of degree $d+1$ in $\CC\PP^d$.
We will use the notation $\delta_p$ for this curve.
We summarise the above discussion in the following proposition that will be implicitly used in the remainder of the paper.
\begin{prop}
An irreducible non-degenerate rational curve of degree $d+1$ in $\CC\PP^d$ is projectively equivalent to $\delta_p$ for some $p\in\CC\PP^{d+1}\setminus\nu_{d+1}$.
\end{prop}
We use the projection point $p$ to define a binary form and a multilinear form associated to $\delta_p$.
The \emph{fundamental binary form} associated to $\delta_p$ is the homogeneous polynomial of degree $d+1$ in two variables $f_p(x,y) := \sum_{i=0}^{d+1}p_i\binom{d+1}{i}x^{d+1-i}y^i$.
Its \emph{polarisation} is the multilinear form $F_p\colon(\FF^2)^{d+1}\to\FF$ \cite{D03}*{Section~1.2} defined by
\[ F_p(x_0,y_0,x_1,y_1,\dots,x_d,y_d) := \frac{1}{(d+1)!}\sum_{I\subseteq\{0,1,\dots,d\}} (-1)^{d+1-|I|} f_p\left(\sum_{i\in I} x_i,\sum_{i\in I} y_i\right). \]
Consider the multilinear form $G_p(x_0,y_0,\dots,x_d,y_d) = \sum_{i=0}^{d+1} p_i P_i$,
where 
\begin{equation}\label{eq:P_i}
P_i(x_0,y_0,x_1,y_1,\dots,x_d,y_d) := \sum_{I\in\binom{\{0,1,\dots,d\}}{i}}\prod_{j\in\overline{I}} x_j\prod_{j\in I} y_j
\end{equation}
for each $i=0,\dots,d+1$.
Here the sum is taken over all subsets $I$ of $\{0,1,\dots,d\}$ of size $i$, and $\overline{I}$ denotes the complement of $I$ in $\{0,1,\dots,d\}$.
It is %also 
easy to see that the binary form $f_p$ is the \emph{restitution} of $G_p$, namely \cite{D03}*{Section~1.2}
\[ f_p(x,y) = G_p(x,y,x,y,\dots,x,y). \]
Since the polarisation of the restitution of a multilinear form is itself \cite{D03}*{Section~1.2},  
we must thus have $F_p = G_p$.
(This can also be checked directly.)

\begin{lemma}\label{lem:cohyperplane}
Let $\delta_p$ be an irreducible non-degenerate rational curve of degree $d+1$ in $\CC\PP^d$, $d\ge 2$, where $p\in\CC\PP^{d+1}\setminus \nu_{d+1}$.
A hyperplane intersects $\delta_p$ in $d+1$ points $\delta_p[x_i,y_i]$, $i=0,\dots,d$, counting multiplicity, if and only if $F_p(x_0,y_0,x_1,y_1,\dots,x_d,y_d)=0$.
\end{lemma}

\begin{proof}
We first prove the statement for distinct points $[x_i,y_i]\in\CC\PP^1$.
Then the points $\delta_p[x_i,y_i]$ are all on a hyperplane if and only if the hyperplane in $\CC\PP^{d+1}$ through the points $\nu_{d+1}[x_i,y_i]$ passes through $p$.
It will be sufficient to prove the identity
\begin{equation}\label{identity}
 D:= \det\begin{pmatrix} \nu_{d+1}[x_0,y_0] \\ \vdots \\  \nu_{d+1}[x_d,y_d] \\ p \end{pmatrix}
= F_p(x_0,y_0,x_1,y_1,\dots,x_d,y_d)
  \prod_{0 \le j<k \le d} \begin{vmatrix} x_j & x_k\\ y_j & y_k \end{vmatrix},
\end{equation}
since the second factor on the right-hand side does not vanish because the points $[x_i, y_i]$ are distinct.
We first note that
\begin{align}
D &=
\begin{vmatrix}
    y_0^{d+1} & -x_0y_0^d & x_0^2y_0^{d-1} & \dotsc & (-x_0)^dy_0 & (-x_0)^{d+1} \\
    \vdots & \vdots & \vdots & \ddots & \vdots & \vdots \\
    y_d^{d+1} & -x_d y_d ^d & x_d ^2y_d ^{d-1} & \dotsc & (-x_d )^dy_d  & (-x_d )^{d+1} \\
    p_0 &  p_1 & p_2 & \dotsc & p_d & p_{d+1}
\end{vmatrix} \notag \\
& = (-1)^{\left \lfloor \frac{d+2}{2} \right \rfloor}
\begin{vmatrix}
    y_0^{d+1} & x_0y_0^d & x_0^2y_0^{d-1} & \dotsc & x_0^d y_0 & x_0^{d+1} \\
    \vdots & \vdots & \vdots & \ddots & \vdots & \vdots \\
    y_d^{d+1} & x_d y_d ^d & x_d ^2y_d ^{d-1} & \dotsc & x_d^dy_d  & x_d^{d+1} \\[3pt]
    p_0 &  -p_1 & p_2 & \dotsc & (-1)^d p_d & (-1)^{d+1} p_{d+1}
\end{vmatrix}. \label{det} 
\end{align}
We next replace $(-1)^i p_i$ by $x^i y^{d+1-i}$ for each $i=0,\dots,d+1$ in the last row of the determinant in \eqref{det} and obtain the Vandermonde determinant
\begin{align*}
\phantom{D} &\mathrel{\phantom{=}}
(-1)^{\left \lfloor \frac{d+2}{2} \right \rfloor}
\begin{vmatrix}
    y_0^{d+1} & x_0y_0^d & x_0^2y_0^{d-1} & \dotsc & x_0^d y_0 & x_0^{d+1} \\
    \vdots & \vdots & \vdots & \ddots & \vdots & \vdots \\
    y_d^{d+1} & x_d y_d^d & x_d^2 y_d^{d-1} & \dotsc & x_d^d y_d  & x_d^{d+1} \\[3pt]
    y^{d+1} &  x y^d & x^2 y^{d-1} & \dotsc & x^d y & x^{d+1}
\end{vmatrix} \\
&=  (-1)^{\left \lfloor \frac{d+2}{2} \right \rfloor}
\prod_{0\le j<k\le d}
\begin{vmatrix}
    y_j & y_k \\
    x_j & x_k
\end{vmatrix}
\prod_{0\le j\le d}
\begin{vmatrix}
    y_j & y \\
    x_j & x
\end{vmatrix} \\
&= (-1)^{\left \lfloor \frac{d+2}{2} \right \rfloor} (-1)^{\binom{d+2}{2}}
\prod_{0\le j<k\le d}
\begin{vmatrix}
    x_j & x_k \\
    y_j & y_k
\end{vmatrix} \prod_{0\le j\le d}
\begin{vmatrix}
    x_j & x \\
    y_j & y
\end{vmatrix}.
\end{align*}
Finally, note that $(-1)^{\lfloor(d+2)/2\rfloor} (-1)^{\binom{d+2}{2}} = 1$ and that the coefficient of $x^i y^{d+1-i}$ in $\prod_{0\le j\le d} \begin{vmatrix} x_j & x \\ y_j & y \end{vmatrix}$ is
\[ \sum_{I\subseteq\binom{\{0,\dots,d\}}{i}}\prod_{j\in I}(-y_j)\prod_{j\in\overline{I}} x_j = (-1)^iP_i, \]
where $P_i$ is as defined in \eqref{eq:P_i}.
It follows that the coefficient of $p_i$ in \eqref{det} is $P_i$, and \eqref{identity} follows.

We next complete the argument for the case when the points $[x_i,y_i]$ are not all distinct.
First suppose that a hyperplane $\Pi$ intersects $\delta_p$ in $\delta_p[x_i,y_i]$, $i=0,\dots,d$.
By Bertini's theorem \cite{H77}*{Theorem~II.8.18 and Remark~II.8.18.1}, there is an arbitrarily close perturbation $\Pi'$ of $\Pi$ that intersects $\delta_p$ in distinct points $\delta_p[x_i',y_i']$.
By what has already been proved, $F_p(x_0',y_0',\dots,x_d',y_d')=0$.
Since $\Pi'$ is arbitrarily close and $F_p$ is continuous, $F_p[x_0,y_0,\dots,x_d,y_d]=0$.

Conversely, suppose that $F_p(x_0,y_0,\dots,x_d,y_d)=0$ where the $[x_i,y_i]$ are not all distinct.
Perturb the points $[x_0,y_0],\dots,[x_{d-1},y_{d-1}]$ by an arbitrarily small amount to $[x_0',y_0'],\dots,[x_{d-1}',y_{d-1}']$ respectively, so as to make $\delta_p[x_0',y_0'],\dots,\delta_p[x_{d-1}',y_{d-1}']$ span a hyperplane $\Pi'$ that intersects $\delta_p$ again in $\delta_p[x_d',y_d']$, say, and so that $[x_0',y_0'],\dots,[x_{d}',y_{d}']$ are all distinct.
If we take the limit as $[x_i',y_i']\to[x_i,y_i]$ for each $i=0,\dots,d-1$, we obtain a hyperplane $\Pi$ intersecting $\delta_p$ in $\delta_p[x_0,y_0],\dots,\delta_p[x_{d-1},y_{d-1}],\delta_p[x_d'',y_d'']$, say.
Then $F_p(x_0,y_0,\dots,x_{d-1},y_{d-1},x_d'',y_d'')=0$.
Since the multilinear form $F_p$ is non-trivial, it follows that $[x_d,y_d]=[x_d'',y_d'']$.
Therefore, $\Pi$ is a hyperplane that intersects $\delta_p$ in $\delta_p[x_i,y_i]$, $i=0,\dots,d$.
\end{proof}

The secant variety $\Sec_{\CC}(\nu_{d+1})$ of the rational normal curve $\nu_{d+1}$ in $\CC\PP^{d+1}$ is equal to the set of points that lie on a proper secant or tangent line of $\nu_{d+1}$, that is, on a line with intersection multiplicity at least $2$ with $\nu_{d+1}$.
We also define the real secant variety of $\nu_{d+1}$ to be the set $\Sec_{\RR}(\nu_{d+1})$ of points in $\RR\PP^{d+1}$ that lie on a line that either intersects $\nu_{d+1}$ in two distinct real points or is a tangent line of $\nu_{d+1}$.
The \emph{tangent variety} $\Tan_{\FF}(\nu_{d+1})$ of $\nu_{d+1}$ is defined to be the set of points in $\FF\PP^{d+1}$ that lie on a tangent line of $\nu_{d+1}$.
We note that although $\Tan_{\RR}(\nu_{d+1}) = \Tan_{\CC}(\nu_{d+1})\cap\RR\PP^{d+1}$, we only have a proper inclusion $\Sec_{\RR}(\nu_{d+1}) \subset \Sec_{\CC}(\nu_{d+1})\cap\RR\PP^{d+1}$ for $d\ge 2$.

We will need a concrete description of $\Sec_{\CC}(\nu_{d+1})$ and its relation to the smoothness of the curves $\delta_p$.
For any $p\in\FF\PP^{d+1}$ and $k=2,\dots,d-1$, define the $(k+1)\times(d-k+2)$ matrix
\begin{equation*}
M_{k}(p) := \begin{pmatrix}
p_0 & p_1 & p_2 & \dots & p_{d-k+1} \\ 
p_1 & p_2 & p_3 & \dots & p_{d-k+2} \\
\vdots & \vdots & \vdots & \ddots & \vdots \\
p_k & p_{k+1} & p_{k+2} & \dots & p_{d+1}
\end{pmatrix}.
\end{equation*}

Suppose that $\delta_p$ has a double point, say $\delta_p[x_0,y_0]=\delta_p[x_1,y_1]$.
This is equivalent to $p$, $\nu_{d+1}[x_0,y_0]$, and $\nu_{d+1}[x_1,y_1]$ being collinear, which is equivalent to $p$ being on the secant variety of $\nu_{d+1}$.
(In the degenerate case where $[x_0,y_0]=[x_1,y_1]$, we have that $p\in\Tan_\FF(\nu_{d+1})$.)
Then $\delta_p[x_0,y_0]$, $\delta_p[x_1,y_1]$, $\delta_p[x_2,y_2]$,\dots, $\delta_p[x_d,y_d]$ are on a hyperplane in $\FF\PP^d$ for all $[x_2,y_2],\dots,[x_d,y_d]\in\FF\PP^1$.
It follows that the coefficients of $F_p(x_0,y_0,x_1,y_1,x_2,y_2,\dots,x_d,y_d)$ as a polynomial in $x_2,y_2,\dots,x_d,y_d$ all vanish, that is,
\[ p_ix_0x_1 + p_{i+1}(x_0y_1+y_0x_1) + p_{i+2}y_0y_1 = 0\]
for all $i=0,\dots,d-1$.
This can be written as
$[x_0x_1, x_0y_1 + y_0x_1, y_0y_1] M_{2}(p) = 0$.
Conversely, if $M_{2}(p)$ has rank $2$ with say 
$[c_0, 2c_1, c_2] M_{2}(p) = 0$,
then there is a non-trivial solution to the linear system with $c_0 = x_0x_1$, $c_1 = x_0y_1 + y_0x_1$, $c_2 = y_0y_1$, and we have $c_0x^2+2c_1xy+c_2y^2 = (x_0x+y_0y)(x_1x+y_1y)$.
In the degenerate case where $[x_0,y_0]=[x_1,y_1]$, we have that the quadratic form has repeated roots.

It follows that $M_{2}(p)$ has rank at most $2$ if and only if $p\in\Sec_{\CC}(\nu_{d+1})$ (also note that $M_{2}(p)$ has rank $1$ if and only if $p\in\nu_{d+1}$).
We note for later use that since the null space of $M_{2}(p)$ is $1$-dimensional if it has rank $2$, it follows that each $p\in\Sec_{\CC}(\nu_{d+1})$ lies on a unique secant (which might degenerate to a tangent).
This implies that $\delta_p$ has a unique singularity when $p\in\Sec_{\CC}(\nu_{d+1})\setminus{\nu_{d+1}}$, which is a node if $p\in\Sec_{\CC}(\nu_{d+1})\setminus\Tan_{\CC}(\nu_{d+1})$ and a cusp if $p\in\Tan_{\CC}(\nu_{d+1})\setminus{\nu_{d+1}}$.
In the real case there are two types of nodes.
If $p\in\Sec_{\RR}(\nu_{d+1})\setminus\nu_{d+1}$, then the roots $[x_0,y_0],[x_1,y_1]$ are real, and $\delta_p$ has either a cusp when $p\in\Tan_{\RR}(\nu_{d+1})\setminus\nu_{d+1}$ and $[x_0,y_0]=[x_1,y_1]$, or a crunode when $p\in\Sec_{\RR}(\nu_{d+1})\setminus\Tan_{\RR}(\nu_{d+1})$ and $[x_0,y_0]$ and $[x_1,y_1]$ are distinct roots of the real binary quadratic form $c_0x^2+2c_1xy+c_2y^2$.
If $p\in\Sec_{\CC}(\nu_{d+1})\setminus\Sec_{\RR}(\nu_{d+1})\cap\RR\PP^{d+1}$ then the quadratic form has conjugate roots $[x_0,y_0]=[\overline{x_1},\overline{y_1}]$ and $\delta_p$ has an acnode.

If $p\notin\Sec(\nu_{d+1})$, then $\delta_p$ is a smooth curve of degree $d+1$.
It follows that $\delta_p$ is singular if and only if $p\in\Sec(\nu_{d+1})\setminus{\nu_{d+1}}$.
For the purposes of this paper, we make the following definitions.
\begin{definition}
A \emph{rational singular curve} is an irreducible non-degenerate singular rational curve of degree $d+1$ in $\CC\PP^d$.
In the real case, a \emph{rational cuspidal curve}, \emph{rational crunodal curve}, or \emph{rational acnodal curve} is a rational singular curve isomorphic to a singular planar cubic with a cusp, crunode, or acnode respectively.
\end{definition}
In particular, we have shown the case $k=2$ of the following well-known result.

\begin{prop}[\cite{H92}*{Proposition~9.7}]\label{prop:secant}
Let $d\ge 3$.
For any $k=2,\dots,d-1$, the secant variety of $\nu_{d+1}$ is equal to the locus of all $[p_0, p_1,\dots,p_{d+1}]$ such that $M_{k}(p)$ has rank at most~$2$.
\end{prop}

\begin{corollary}\label{lem:first_species}
Let $d\ge 3$.
For any $k=2,\dots,d-1$ and $p\in\CC\PP^{d+1}\setminus\nu_{d+1}$, the curve $\delta_p$ of degree $d+1$ in $\CC\PP^d$ is singular if and only if $\rank M_{k}(p) \le 2$.
\end{corollary}

We next use Corollary~\ref{lem:first_species} to show that the projection of a smooth rational curve of degree $d+1$ in $\CC\PP^d$ from a generic point on the curve is again smooth when $d\ge 4$.
This is not true for $d=3$, as there is a trisecant through each point of a quartic curve of the second species in $3$-space.
(The union of the trisecants form the unique quadric on which the curve lies \cite{H92}*{Exercise 8.13}.)

\begin{lemma}\label{lemma:singular_projection}
Let $\delta_p$ be a smooth rational curve of degree $d+1$ in $\CC\PP^d$, $d\ge 4$.
Then for all but at most three points $q\in\delta_p$, the projection $\overline{\pi_q(\delta_p\setminus\{q\})}$ is a smooth rational curve of degree $d$ in $\CC\PP^{d-1}$.
\end{lemma}

\begin{proof}
Let $q=\delta_p[x_0,y_0]$.
Suppose that $\overline{\pi_q(\delta_p\setminus\{q\})}$ is singular.
Then there exist $[x_1,y_1]$ and $[x_2,y_2]$ such that $\pi_q(\delta_p[x_1,y_1])=\pi_q(\delta_p[x_2,y_2])$
and the points $\delta_p[x_0,y_0]$, $\delta_p[x_1,y_1]$, and $\delta_p[x_2,y_2]$ are collinear.
Then for arbitrary $[x_3,y_3],\dots,[x_d,y_d]\in\CC\PP^1$, the points $\delta_p[x_i,y_i]$, $i=0,\dots,d$ are on a hyperplane, so by Lemma~\ref{lem:cohyperplane}, $F_p(x_0,y_0,\dots,x_d,y_d)$ is identically $0$ as a polynomial in $x_3,y_3,\dots,x_d,y_d$.
The coefficients of this polynomial are of the form
\[ p_i x_0x_1x_2 + p_{i+1}(x_0x_1y_2+x_0y_1x_2+y_0x_1x_2) + p_{i+2}(x_0y_1y_2+y_0x_1y_2+y_0y_1x_2) + p_{i+3} y_0y_1y_2 \]
for $i=0,\dots,d-2$.
This means that the linear system
$[c_0, 3c_1, 3c_2, c_3] M_3(p) = 0$ 
has a non-trivial solution $c_0=x_0x_1x_2$, $3c_1=x_0x_1y_2+x_0y_1x_2+y_0x_1x_2$, $3c_2=x_0y_1y_2+y_0x_1y_2+y_0y_1x_2$, $c_3=y_0y_1y_2$.
The binary cubic form $c_0 x^3 + 3c_1 x^2y + c_2 xy^2 + c_3 y^3$ then has the factorisation $(x_0x+y_0y)(x_1x+y_1y)(x_2x+y_2y)$,
hence its roots give the collinear points on $\delta_p$.
Since $\delta_p$ is smooth, $M_3(p)$ has rank at least $3$ by Corollary~\ref{lem:first_species}, and so the cubic form is unique up to scalar multiples.
It follows that there are at most three points $q$ such that the projection $\overline{\pi_q(\delta_p\setminus\{q\})}$ is not smooth.
\end{proof}

We need the following theorem on the fundamental binary form $f_p$ that is essentially due to Sylvester \cite{S51} to determine the natural group structure on rational singular curves.
Reznick \cite{Rez2013} gives an elementary proof of the generic case where $p$ does not lie on the tangent variety.
(See also Kanev \cite{K99}*{Lemma~3.1} and Iarrobino and Kanev \cite{IK99}*{Section~1.3}.)
We provide a very elementary proof that includes the non-generic case.

\begin{theorem}[Sylvester \cite{S51}]\label{thm:sylvester}
Let $d\ge 2$.
\begin{enumerate}[label=\rom]
\item \label{sylvester1}
If $p\in\Tan_{\CC}(\nu_{d+1})$, then there exist binary linear forms $L_1,L_2$ such that $f_p(x,y)=L_1(x,y)^dL_2(x,y)$.
Moreover, if $p \notin \nu_{d+1}$ then $L_1$ and $L_2$ are linearly independent,
and if $p\in\RR\PP^{d+1}$ then $L_1$ and $L_2$ are both real.
\item\label{sylvester2}
If $p\in\Sec_{\CC}(\nu_{d+1})\setminus\Tan_{\CC}(\nu_{d+1})$, then there exist linearly independent binary linear forms $L_1, L_2$ such that $f_p(x,y) = L_1(x,y)^{d+1} - L_2(x,y)^{d+1}$.
Moreover, if $p\in\RR\PP^{d+1}\setminus\Sec_{\RR}(\nu_{d+1})$ then $L_1$ and $L_2$ are complex conjugates, while if $p\in\Sec_{\RR}(\nu_{d+1})$ then there exist linearly independent real binary linear forms $L_1, L_2$ such that $f_p(x,y) = L_1(x,y)^{d+1} \pm L_2(x,y)^{d+1}$, where we can always choose the lower sign when $d$ is even, and otherwise depends on $p$.
\end{enumerate}
\end{theorem}

\begin{proof}
\ref{sylvester1}:
We work over $\FF\in\{\RR,\CC\}$.
Let $p = [p_0,p_1,\dots,p_{d+1}]\in \Tan_\FF(\nu_{d+1})$.
Let $p_*=\nu_{d+1}[\alpha_1,\alpha_2]$ be the point on $\nu_{d+1}$ such that the line $pp_*$ is tangent to $\nu_{d+1}$ (if $p\in\nu_{d+1}$, we let $p_*=p$).
We will show that
\begin{equation}\label{tangent}
f_p(x,y) = \sum_{i=0}^{d+1} p_i\binom{d+1}{i}x^{d+1-i}y^i = (\alpha_2 x - \alpha_1 y)^d(\beta_2 x-\beta_1 y)
\end{equation}
for some $[\beta_1,\beta_2]\in\FF\PP^1$.

First consider the special case $\alpha_1=0$.
Then $p_*=[1,0,\dots,0]$ and the tangent to $\nu_{d+1}$ at $p_*$ is the line $x_2=x_3=\dots=x_{d+1}=0$.
It follows that $f_p(x,y)= p_0x^{d+1} + p_1(d+1)x^dy = (1x-0y)^d(p_0x+p_1(d+1)y)$.
If $p_1 = 0$, then $p = p_* \in \nu_{d+1}$.
Thus, if $p \notin \nu_{d+1}$, then $p_1 \ne 0$, and $x$ and $p_0x+p_1(d+1)y$ are linearly independent.

We next consider the general case $\alpha_1\neq 0$.
Equating coefficients in \eqref{tangent}, we see that we need to find $[\beta_1,\beta_2]$ such that
\[ p_i\binom{d+1}{i} = \binom{d}{i}\alpha_2^{d-i}(-\alpha_1)^i\beta_2-\binom{d}{i-1}\alpha_2^{d-i+1}(-\alpha_1)^{i-1}\beta_1 \]
for each $i=0,\dots,d+1$, where we use the convention $\binom{d}{-1}=\binom{d}{d+1}=0$.
This can be simplified to
\begin{equation}\label{tangent2}
p_i = \left(1-\frac{i}{d+1}\right)\alpha_2^{d-i}(-\alpha_1)^i\beta_2 - \frac{i}{d+1}\alpha_2^{d-i+1}(-\alpha_1)^{i-1}\beta_1.
\end{equation}
Since we are working projectively, we can fix the value of $\beta_1$ from the instance $i=d+1$ of~\eqref{tangent2} to get
\begin{equation}\label{tangent3}
p_{d+1} = -(-\alpha_1)^d\beta_1.
\end{equation}

If $p_{d+1}\neq 0$, we can divide \eqref{tangent2} by \eqref{tangent3}. 
After setting $\alpha=\alpha_2/\alpha_1$, $\beta=\beta_2/\beta_1$, and $a_i=p_i/p_{d+1}$, we then have to show that for some $\beta\in\FF$,
\begin{equation}\label{eqn:*}
a_i = -\left(1-\frac{i}{d+1}\right)(-\alpha)^{d-i}\beta+\frac{i}{d+1}(-\alpha)^{d-i+1}
\end{equation}
for each $i=0,\dots,d$.
We next calculate in the affine chart $x_{d+1}=1$ where the rational normal curve becomes $\nu_{d+1}(t) = ((-t)^{d+1},(-t)^d,\dots,-t)$, $p=(a_0,\dots,a_d)$, and $p_*=\nu_{d+1}(\alpha)$.
The tangency condition means that $p_*-p$ is a scalar multiple of 
\[\nu_{d+1}'(\alpha) = ((d+1)(-\alpha)^d,d(-\alpha)^{d-1},\dots,2\alpha,-1), \] 
that is, we have for some $\lambda\in\FF$ that
$(-\alpha)^{d+1-i}-a_i = \lambda(d+1-i)(-\alpha)^{d-i}$ for all $i=0,\dots,d$.
Set $\beta=\alpha+\lambda(d+1)$.
Then $(-\alpha)^{d+1-i}-a_i = (\beta-\alpha)(1-\frac{i}{d+1})(-\alpha)^{d-i}$, and we have
\begin{align*}
a_i &= (-\alpha)^{d+1-i}-(\beta-\alpha)\left(1-\frac{i}{d+1}\right)(-\alpha)^{d-i} \\
&= -\left(1-\frac{i}{d+1}\right)(-\alpha)^{d-i}\beta+\frac{i}{d+1}(-\alpha)^{d-i+1},
\end{align*}
giving \eqref{eqn:*} as required.
If $\alpha = \beta$, then $\lambda = 0$ and $p = p_* \in \nu_{d+1}$.
Thus, if $p \notin \nu_{d+1}$, then $\alpha \ne \beta$, and $\alpha_2 x - \alpha_1 y$ and $\beta_2 x - \beta_1 y$ are linearly independent.

We still have to consider the case $p_{d+1}=0$.
Then $\beta_1=0$ and we need to find $\beta_2$ such that
\begin{equation}\label{eqn:**}
p_i = \left(1-\frac{i}{d+1}\right)\alpha_2^{d-i}(-\alpha_1)^i\beta_2
\end{equation}
for all $i=0,\dots,d$.
Since $p_{d+1}=0$, we have that $\nu_{d+1}'(\alpha)$ is parallel to $(p_0,\dots,p_d)$, that is,
\[ p_i = \lambda(d+1-i)(-\alpha)^{d-i}\]
for some $\lambda\in\FF^*$.
Set $\beta_2 = \lambda(d+1)/(-\alpha_1)^d$.
Then
%\begin{align*}
%p_i &= \frac{(-\alpha_1)^d\beta_2}{d+1}(d+1-i)\left(\frac{\alpha_2}{-\alpha_1}\right)^{d-i}\\
%&= \left(1-\frac{i}{d+1}\right)\alpha_2^{d-i}(-\alpha_1)^i\beta_2,
%\end{align*}
\begin{equation*}
p_i = \frac{(-\alpha_1)^d\beta_2}{d+1}(d+1-i)\left(\frac{\alpha_2}{-\alpha_1}\right)^{d-i}
= \left(1-\frac{i}{d+1}\right)\alpha_2^{d-i}(-\alpha_1)^i\beta_2,
\end{equation*}
again giving \eqref{eqn:**} as required.
Note that since $\alpha_1 \ne 0$ but $\beta_1 = 0$, $\alpha_2 x - \alpha_1 y$ and $\beta_2 x - \beta_1 y$ are linearly independent.
Note also that since $\lambda \ne 0$, we have $\beta_2 \ne 0$ and $p \ne [1,0,\dotsc,0]$, hence $p \notin \nu_{d+1}$.

\ref{sylvester2}:
Let $p=[p_0,\dots,p_{d+1}] \in \Sec_{\CC}(\nu_{d+1})\setminus\Tan_{\CC}(\nu_{d+1})$, and suppose that $p$ lies on the secant line through the distinct points $p_1 := \nu_{d+1}[\alpha_1, \alpha_2]$ and $p_2 := \nu_{d+1}[\beta_1, \beta_2]$.
Since $p, p_1, p_2$ are distinct and collinear, there exist $\mu_1, \mu_2\in\CC^*$ such that $p = \mu_1 p_1 + \mu_2 p_2$.
This means that for $i = 0, \dotsc, d+1$, we have
\[ p_i = \mu_1(-\alpha_1)^i\alpha_2^{d+1-i} + \mu_2(-\beta_1)^i\beta_2^{d+1-i}. \]
Then
\begin{align*}
f_p(x,y) &= \sum_{i=0}^{d+1} p_i \binom{d+1}{i} x^{d+1-i} y^{i} \\
&= \mu_1\sum_{i=0}^{d+1}\binom{d+1}{i}(\alpha_2 x)^{d+1-i}(-\alpha_1 y)^i + \mu_2 \sum_{i=0}^{d+1}\binom{d+1}{i}(\beta_2 x)^{d+1-i}(-\beta_1 y)^i \\
&= \mu_1(\alpha_2 x-\alpha_1 y)^{d+1} + \mu_2(\beta_2 x-\beta_1 y)^{d+1}\\
&= L_1(x,y)^{d+1}- L_2(x,y)^{d+1}
\end{align*}
where the linear forms $L_1,L_2$ are linearly independent.

If $p\in\RR\PP^{d+1}\setminus{\Sec_{\RR}(\nu_{d+1})}$, then $f_p$ is real and $p_1$ and $p_2$ are non-real points.
Taking conjugates, we have
\[ p = \overline{\mu_1}\nu_{d+1}[\overline{\alpha_1},\overline{\alpha_2}]+\overline{\mu_2}\nu_{d+1}[\overline{\beta_1},\overline{\beta_2}] \]
as vectors, and because of the uniqueness of secants of the rational normal curve through a given point, we obtain $\overline{\mu_1}=\mu_2$ and $\nu_{d+1}[\overline{\alpha_1},\overline{\alpha_2}] = \nu_{d+1}[\beta_1,\beta_2]$, hence $\overline{\alpha_1}=\beta_1$ and $\overline{\alpha_2}=\beta_2$.
It follows that $\overline{L_1(x,y)} = L_2(\overline{x},\overline{y})$.

If $p\in\Sec_{\RR}(\nu_{d+1})$, then $p_1$ and $p_2$ are real, so $[\mu_1,\mu_2],[\alpha_1,\alpha_2],[\beta_1,\beta_2]\in\RR\PP^1$, and we obtain $f_p(x,y)=L_1^{d+1}\pm L_2^{d+1}$ for some linearly independent $L_1,L_2$ over $\RR$, where the choice of sign depends on~$p$.
\end{proof}

We are now in a position to describe the group laws on rational singular curves.
We first note the effect of a change of coordinates on the parametrisation of $\delta_p$.
Let $\varphi\colon\FF\PP^1\to\FF\PP^1$ be a projective transformation. 
Then $\nu_{d+1}\circ\varphi$ is a reparametrisation of the rational normal curve.
It is not difficult to see that there exists a projective transformation $\psi\colon\FF\PP^{d+1}\to\FF\PP^{d+1}$ such that $\nu_{d+1}\circ\varphi = \psi\circ\nu_{d+1}$.
It follows that if we reparametrise $\delta_p$ using $\varphi$, we obtain
\[ \delta_p\circ\varphi = \pi_p\circ\nu_{d+1}\circ\varphi = \pi_p\circ\psi\circ\nu_{d+1} = \psi'\circ\pi_{\psi^{-1}(p)}\circ\nu_{d+1}\cong\delta_{\psi^{-1}(p)}, \]
where $\psi'\colon\FF\PP^d\to\FF\PP^d$ is an appropriate projective transformation such that first transforming $\FF\PP^{d+1}$ with $\psi$ and then projecting from $p$ is the same as projecting from $\psi^{-1}(p)$ and then transforming $\FF\PP^d$ with $\psi'$.
So by reparametrising $\delta_p$, we obtain $\delta_{p'}$ for some other point $p'$ that is in the orbit of $p$ under the action of projective transformations that fix $\nu_{d+1}$.

Since $\delta_p\circ\varphi[x_0,y_0],\dots,\delta_p\circ\varphi[x_d,y_d]$ lie on a hyperplane if and only if the $\delta_{\psi^{-1}(p)}[x_i,y_i]$'s
are on a hyperplane, it follows from Lemma~\ref{lem:cohyperplane} that $F_p(\varphi(x_0,y_0),\dots,\varphi(x_d,y_d))$ is a scalar multiple of $F_{\psi^{-1}(p)}(x_0,y_0,\dots,x_d,y_d)$, in which case $f_p\circ\varphi = f_{\psi^{-1}(p)}$ up to a scalar multiple.
Thus, we obtain the same reparametrisation of the fundamental binary form $f_p$.

\begin{prop}\label{prop:rational_group}
A rational singular curve $\delta_p$ in $\CC\PP^d$ has a natural group structure on its subset of smooth points $\delta_p^*$ such that $d+1$ points in $\delta_p^*$ lie on a hyperplane if and only if they sum to the identity.
This group is isomorphic to $(\CC,+)$ if the singularity of $\delta_p$ is a cusp and isomorphic to $(\CC^*,\cdot)$ if the singularity is a node.

If the curve is real and cuspidal or acnodal, then it has a group isomorphic to $(\RR,+)$ or $\RR/\ZZ$ depending on whether the singularity is a cusp or an acnode, such that $d+1$ points in $\delta_p^*$ lie on a hyperplane if and only if they sum to the identity.
If the curve is real and the singularity is a crunode, then the group is isomorphic to $(\RR,+) \times \ZZ_2$, but $d+1$ points in $\delta_p^*$ lie on a hyperplane if and only if they sum to $(0,0)$ or $(0,1)$, depending on $p$.
\end{prop}

\begin{proof}
First suppose $\delta_p$ is cuspidal and $\FF\in\{\RR,\CC\}$, so that $p \in \Tan_{\FF}(\nu_{d+1}) \setminus{\nu_{d+1}}$.
By Theorem~\ref{thm:sylvester}, $f_p=L_1^dL_2$ for some linearly independent linear forms $L_1$ and $L_2$.
By choosing $\varphi$ appropriately, we may assume without loss of generality that $L_1(x,y)=x$ and $L_2(x,y)=(d+1)y$, so that $f_p(x,y)=(d+1)x^dy$ and $p=[0,1,0,\dots,0]$, with the cusp of $\delta_p$ at $\delta_p[0,1]$.
It follows that the polarisation of $f_p$ is $F_p(x_0,y_0,\dotsc,x_d,y_d)= P_1 = x_0x_1\dotsb x_d\sum_{i=0}^d y_i/x_i$.
For $[x_i,y_i]\neq[0,1]$, $i=0,\dots,d$, the points $\delta_p[x_i,y_i]$ are on a hyperplane if and only if $\sum_{i=0}^d y_i/x_i=0$.
Thus we identify $\delta_p[x,y]\in\delta_p^*$ with $y/x\in\FF$, and the group is $(\FF,+)$.

Next suppose $\delta_p$ is nodal, so that $p \in \Sec_{\CC}(\nu_{d+1})\setminus\Tan_{\CC}(\nu_{d+1})$.
By Theorem~\ref{thm:sylvester}, $f_p=L_1^{d+1}-L_2^{d+1}$ for some linearly independent linear forms $L_1$ and $L_2$.
Again by choosing $\varphi$ appropriately, we may assume without loss of generality that $L_1(x,y)=x$ and $L_2(x,y)=y$, so that $f_p(x,y)=x^{d+1}-y^{d+1}$ and $p=[1,0,\dots,0,-1]$, with the node of $\delta_p$ at $\delta_p[0,1]=\delta_p[1,0]$.
The polarisation of $f_p$ is $F_p(x_0,y_0,\dots,x_d,y_d)= P_0-P_{d+1} = x_0x_1\dotsb x_d - y_0y_1\dotsb y_d$.
Therefore, $\delta_p[x_i,y_i]$, $i=0,\dotsc,d$, are on a hyperplane if and only if $\prod_{i=0}^d y_i/x_i = 1$.
Thus we identify $\delta_p[x,y]\in\delta_p^*$ with $y/x\in\CC^*$, and the group is $(\CC^*,\cdot)$.

Now suppose $\delta_p$ is real and the node is an acnode.
Then the linearly independent linear forms $L_1$ and $L_2$ given by Theorem~\ref{thm:sylvester} are $L_1(x,y) = \alpha x+ \beta y$ and $L_2(x,y)=\overline{\alpha} x + \overline{\beta} y$ for some $\alpha,\beta\in\CC\setminus\RR$.
There exists $\varphi\colon\RR\PP^1\to\RR\PP^1$ such that $L_1\circ\varphi = x + iy$ and $L_2\circ\varphi = x-iy$, hence we may assume after such a reparametrisation that $f_p(x,y) = (x+iy)^{d+1} - (x-iy)^{d+1}$ and that the node is at $\delta_p[i,1]=\delta_p[-i,1]$.
The polarisation of $f_p$ is $F_p(x_0,y_0,\dots,x_d,y_d) = \prod_{j=0}^d(x_j+iy_j) - \prod_{j=0}^d(x_j-iy_j)$, and it follows that $\delta_p[x_0,y_0], \dotsc, \delta_p[x_d,y_d]$ are collinear if and only if $\prod_{j=0}^d\frac{x_j+iy_j}{x_j-iy_j} = 1$.
We now identify $\RR\PP^1$ with the circle $\RR/\ZZ \cong \setbuilder{z \in \CC}{|z|=1}$ using the M\"obius transformation $[x,y]\to \frac{x+iy}{x-iy}$.

It remains to consider the crunodal case.
Then, similar to the complex nodal case, we obtain after a reparametrisation that $\delta_p[x_i,y_i]$, $i = 0, \dotsc, d$, are on a hyperplane if and only if $\prod_{i=0}^d y_i/x_i = \pm 1$, where the sign depends on $p$.
Thus we identify $\delta_p[x,y]\in\delta_p^*$ with $y/x\in\RR^*$, and the group is $(\RR^*,\cdot)\cong\RR\times\ZZ_2$, where $\pm 1\in\RR^*$ corresponds to $(0,0),(0,1)\in\RR\times\ZZ_2$ respectively.
\end{proof}

The group on an elliptic normal curve or a rational singular curve of degree $d+1$ as described in Propositions~\ref{prop:elliptic_group} and \ref{prop:rational_group} is not uniquely determined by the property that $d+1$ points lie on a hyperplane if and only if they sum to some fixed element $c$.
Indeed, for any $t\in(\delta^*,\oplus)$, $x\boxplus y:= x\oplus y\oplus t$ defines another abelian group on $\delta^*$ with the property that $d+1$ points lie on a hyperplane if and only if they sum to $c\oplus dt$.
However, these two groups are isomorphic in a natural way with an isomorphism given by the translation map $x\mapsto x\ominus t$.
The next proposition show that we always get uniqueness up to some translation.
It will be used in Section~\ref{sec:extremal}.

\begin{prop}\label{prop:unique}
Let $(G,\oplus,0)$ and $(G,\boxplus,0')$ be abelian groups on the same ground set, such that for some $d\ge 2$ and some $c,c'\in G$,
\[ x_1\oplus\dotsb\oplus x_{d+1} =c \iff x_1\boxplus\dotsb\boxplus x_{d+1}=c'\quad\text{for all }x_1,\dots,x_{d+1}\in G.\]
Then $(G,\oplus,0) \to (G,\boxplus,0'), x\mapsto x\boxminus 0 = x \oplus 0'$ is an isomorphism, and 
\[ c'=c\boxplus \underbrace{0 \boxplus \dotsb \boxplus 0}_\text{$d$ times} = c \ominus (\underbrace{0' \oplus \dotsb \oplus 0'}_\text{$d$ times}). \]
\end{prop}

\begin{proof}
It is clear that the cases $d\ge 3$ follow from the case $d=2$, which we now show.
First note that for any $x,y\in G$, $x\boxplus y\boxplus(c \ominus x \ominus y) = c'$ and $(x \oplus y)\boxplus 0\boxplus(c \ominus x \ominus y) = c'$, since $x \oplus y \oplus (c \ominus x \ominus y)=(x \oplus y) \oplus 0 \oplus (c \ominus x \ominus y)=c$.
Thus we have $x\boxplus y = (x \oplus y)\boxplus 0$, hence
$(x \oplus y) \boxminus 0 = x \boxplus y \boxminus 0 \boxminus 0 = (x \boxminus 0) \boxplus (y \boxminus 0)$.
Similarly we have $x \oplus y=(x\boxplus y) \oplus 0'$, hence $x \boxplus y = x \oplus y \ominus 0'$,
so in particular
$0' = 0 \boxminus 0 = 0 \oplus (\boxminus 0) \ominus 0'$, and $\boxminus 0 = 0' \oplus 0'$.
So we also have $x\boxminus 0 = x \oplus (\boxminus 0) \ominus 0' = x \oplus 0'$, 
and $(G,\oplus,0) \to (G,\boxplus,0'), x\mapsto x\boxminus 0 = x \oplus 0'$ is an isomorphism.
\end{proof}

\section{Structure theorem}\label{sec:proof}

We prove Theorem~\ref{thm:main1} in this section. 
The main idea is to induct on the dimension $d$ via projection.
We start with the following statement of the slightly different case $d=3$, which is \cite{LS18}*{Theorem~1.1}.
Note that it contains one more type that does not occur when $d\ge 4$.

\begin{theorem}\label{thm:plane}
Let $K > 0$ and suppose $n \ge C\max\{K^8,1\}$ for some sufficiently large absolute constant $C > 0$.
Let $P$ be a set of $n$ points in $\RR\PP^3$ with no $3$ points collinear.
If $P$ spans at most $Kn^2$ ordinary planes,
then up to projective transformations, $P$ differs in at most $O(K)$ points from a configuration of one of the following types:
\begin{enumerate}[label=\rom]
\item A subset of a plane;
\item A subset of two disjoint conics lying on the same quadric with $\frac{n}{2}\pm O(K)$ points of $P$ on each of the two conics;
\item A coset of a subgroup of the smooth points of an elliptic or acnodal space quartic curve.
\end{enumerate}
\end{theorem}

We first prove the following weaker lemma using results from Section~\ref{sec:tools}.

\begin{lemma}\label{lem:intermediate}
Let $d \ge 4$, $K > 0$, and suppose $n \ge C\max\{d^3 2^dK, (dK)^8\}$ for some sufficiently large absolute constant $C > 0$. 
Let $P$ be a set of $n$ points in $\RR\PP^d$ where every $d$ points span a hyperplane.
If $P$ spans at most $K\binom{n-1}{d-1}$ ordinary hyperplanes, then all but at most $O(d2^dK)$ points of $P$ are contained in a hyperplane or an irreducible non-degenerate curve of degree $d+1$ that is either elliptic or rational and singular.
\end{lemma}

\begin{proof}
We use induction on $d \ge 4$ to show that for all $K>0$ and all $n \ge f(d,K)$,
for all sets $P$ of $n$ points in $\RR\PP^d$ with any $d$ points spanning a hyperplane,
if $P$ has at most $K\binom{n-1}{d-1}$ ordinary hyperplanes,
then all but at most $g(d,K)$ points of $P$ are contained in a hyperplane or an irreducible non-degenerate curve of degree $d+1$, and that if the curve is rational then it has to be singular, where
\[ g(d,K) := \sum_{k=0}^d k^3 2^{d-k} + C_12^d(d-1)K \]
and
\[ f(d,K) := d^2(g(d,K) + C_2 d^{10}) + C(d-1)^8K^8 \]
for appropriate $C_1,C_2>0$ to be determined later and $C$ from Theorem~\ref{thm:plane}.
We assume that this holds in $\RR\PP^{d-1}$ if $d \ge 5$, while Theorem~\ref{thm:plane} takes the place of the induction hypothesis when $d=4$.

Let $P'$ denote the set of points $p \in P$ such that there are at most $\frac{d-1}{d-2}K\binom{n-2}{d-2}$ ordinary hyperplanes through $p$.
By counting incident point-ordinary-hyperplane pairs, we obtain
\[dK\binom{n-1}{d-1}  > (n-|P'|)\frac{d-1}{d-2}K\binom{n-2}{d-2},\]
which gives $|P'| > n/(d-1)^2$.
For any $p \in P'$, the projected set $\pi_p(P \setminus \{p\})$ has $n-1$ points and spans at most $\frac{d-1}{d-2}K\binom{n-2}{d-2}$ ordinary $(d-2)$-flats in $\RR\PP^{d-1}$, and any $d-1$ points of $\pi_p(P \setminus \{p\})$ span a $(d-2)$-flat.
To apply the induction hypothesis, we need
\[f(d,K) \ge 1+ f(d-1,\tfrac{d-1}{d-2}K),\]
as well as $f(3,K)\ge C\max\{K^8,1\}$, both of which easily follow from the definition of $f(d,K)$.
Then all except $g(d-1,\frac{d-1}{d-2}K)$
points of $\pi_p(P \setminus \{p\})$ are contained in a $(d-2)$-flat or a non-degenerate curve $\gamma_p$ of degree $d$ in $\RR\PP^{d-1}$, which is either irreducible or possibly two conics with $\frac{n}{2} \pm O(K)$ points on each when $d=4$.

If there exists a $p \in P'$ such that all but at most $g(d-1,\frac{d-1}{d-2}K)$
points of $\pi_p(P \setminus \{p\})$ are contained in a $(d-2)$-flat, then we are done, since
$g(d,K) > g(d-1,\frac{d-1}{d-2}K)$.
Thus we may assume without loss of generality that for all $p\in P'$ we obtain a curve $\gamma_p$.

Let $p$ and $p'$ be two distinct points of $P'$.
Then all but at most $2g(d-1,\frac{d-1}{d-2}K)$ points of $P$ lie on the intersection $\delta$ of the two cones $\overline{\pi^{-1}_p(\gamma_p)}$ and $\overline{\pi^{-1}_{p'}(\gamma_{p'})}$.
Since the curves $\gamma_p$ and $\gamma_{p'}$ are $1$-dimensional, the two cones are $2$-dimensional.
Since their vertices $p$ and $p'$ are distinct, the cones do not have a common irreducible component, so their intersection is a variety of dimension at most~$1$.
By B\'ezout's theorem (Theorem~\ref{thm:bezout}),
$\delta$ has total degree at most $d^2$, so has to have at least one $1$-dimensional irreducible component.
Let $\delta_1, \dotsc, \delta_k$ be the $1$-dimensional components of $\delta$, where $1\le k \le d^2$.
Let $\delta_1$ be the component with the most points of $P'$ amongst all the $\delta_i$, so that
\[|P' \cap \delta_1| \ge \frac{|P'|-2g(d-1,\frac{d-1}{d-2}K)}{d^2}.\]
Choose a $q \in P' \cap \delta_1$ such that $\pi_q$ is generically one-to-one on $\delta_1$. 
By Lemma~\ref{lem:projection1} there are at most $O(\deg(\delta_1)^4)=O(d^8)$ exceptional points, so we need
\begin{equation}\label{constraint1}
|P'\cap\delta_1|> C_2 d^8.
\end{equation}
Since $|P'| > n/(d-1)^2$, we need
\[ \frac{\frac{n}{(d-1)^2}-2g(d-1,\frac{d-1}{d-2}K)}{d^2} > C_2 d^8,\]
or equivalently, $n > (d-1)^2(2g(d-1,\frac{d-1}{d-2}K)+C_2 d^{10})$.
However, this follows from the definition of $f(d,K)$.
If $\pi_q$ does not map $\delta_1 \setminus \{q\}$ into $\gamma_q$, then by B\'ezout's theorem (Theorem~\ref{thm:bezout}), $n-1-g(d-1,\binom{d-1}{d-2}K)\le d^3$.
However, this does not occur since $f(d,K) > g(d-1,\binom{d-1}{d-2}K) + d^3 + 1$.
Thus, $\pi_q$ maps $\delta_1\setminus\{q\}$ into $\gamma_q$, hence $\delta_1$ is an irreducible curve of degree $d+1$ (or, when $d=4$, possibly a twisted cubic containing at most $n/2+O(K)$ points of $P$).

We first consider the case where $\delta_1$ has degree $d+1$.
We apply Lemma~\ref{lem:projection3} to $\delta_1$ and each $\delta_i$, $i=2,\dots,k$, and for this we need $|P'\cap\delta_1| > C'' d^4$, since $\deg(\delta_1)\le d^2$ and $\sum_{i=2}^d\deg(\delta_i)\le d^2$.
However, this condition is implied by \eqref{constraint1}.
Thus we find a $q' \in P' \cap \delta_1$ such that $\overline{\pi_{q'}(\delta_1\setminus\{q'\})}=\gamma_{q'}$ as before, and in addition, the cone $\overline{\pi_{q'}^{-1}(\gamma_{q'})}$ does not contain any other $\delta_i$, $i=2,\dots,k$.
Since all points of $P$ except $2g(d-1,\frac{d-1}{d-2}K)+d^2$ lie on $\delta_1\cup\dots\cup\delta_k$, we obtain by B\'ezout's theorem (Theorem~\ref{thm:bezout}) that
\[|P\setminus{\delta_1}| \le d(d^2-d-1) + d^2 + 2g(d-1,\tfrac{d-1}{d-2}K) < g(d,K).\]

We next dismiss the case where $d=4$ and $\delta_1$ is a twisted cubic.
We redefine $P'$  to be the set of points $p\in P$ such that there are at most $12Kn^2$ ordinary hyperplanes through $p$.
Then $|P'|\ge 2n/3$.
Since we have $|P\cap\delta_1|\le n/2 +O(K)$, by Lemma~\ref{lem:projection2} there exists $q'\in P'\setminus\delta_1$ such that the projection from $q'$ will map $\delta_1$ onto a twisted cubic in $\RR\PP^3$.
However, by B\'ezout's theorem (Theorem~\ref{thm:bezout}) and Theorem~\ref{thm:plane}, $\pi_{q'}(\delta_1\setminus\{q'\})$ has to be mapped onto a conic, which gives a contradiction.

Note that $g(d,K)=O(d 2^d K)$ since $K = \Omega(1/d)$ by \cite{BM17}*{Theorem 2.4}.
We have shown that all but $O(d2^dK)$ points of $P$ are contained in a hyperplane or an irreducible non-degenerate curve $\delta$ of degree $d+1$.
By Proposition~\ref{prop:curves}, this curve is either elliptic or rational.
It remains to show that if $\delta$ is rational, then it has to be singular.
Similar to what was shown above, we can find more than $3$ points $p\in\delta$ for which the projection $\overline{\pi_p(\delta\setminus\{p\})}$ is a rational curve of degree $d$ that is singular by the induction hypothesis.
Lemma~\ref{lemma:singular_projection} now implies that $\delta$ is singular.
\end{proof}

To get the coset structure on the curves as stated in Theorem~\ref{thm:main1}, we use a simple generalisation of an additive combinatorial result used by Green and Tao \cite{GT13}*{Proposition A.5}.
This captures the principle that if a finite subset of a group is almost closed, then it is close to a subgroup. 
The case $d=3$ was shown in~\cite{LMMSSZ18}.

\begin{lemma}\label{cor:a5}
Let $d \ge 2$.
Let $A_1, A_2, \dotsc, A_{d+1}$ be $d+1$ subsets of some abelian group $(G,\oplus)$, all of size within $K$ of $n$, where $K \le cn/d^2$ for some sufficiently small absolute constant $c > 0$.
Suppose there are at most $Kn^{d-1}$ $d$-tuples $(a_1,a_2,\dotsc,a_d) \in A_1 \times A_2 \times \dotsb \times A_d$ for which $a_1 \oplus a_2 \oplus \dotsb \oplus a_d \notin A_{d+1}$. Then there is a subgroup $H$ of $G$ and cosets $H \oplus x_i$ for $i = 1, \dotsc, d$ such that 
\begin{equation*}
|A_i \tri (H \oplus x_i)|, \left| A_{d+1} \tri \left( H \oplus \bigoplus_{i=1}^d x_i \right) \right| = O(K).
\end{equation*}
\end{lemma}

\begin{proof}
We use induction on $d \ge 2$ to show that the symmetric differences in the conclusion of the lemma have size at most $C \prod_{i=1}^d (1+\frac{1}{i^2})K$ for some sufficiently large absolute constant $C > 0$.
The base case $d=2$ is \cite{GT13}*{Proposition A.5}.

Fix a $d \ge 3$.
By the pigeonhole principle, there exists $b_1 \in A_1$ such that there are at most
\[ \frac{1}{n-K}Kn^{d-1} \le \frac{1}{1-\frac{c}{d^2}} Kn^{d-2} \]
$(d-1)$-tuples $(a_2, \dotsc, a_d) \in A_2 \times \dotsb \times A_d$ for which $b_1 \oplus a_2 \oplus \dotsb \oplus a_d \notin A_{d+1}$, or equivalently $a_2 \oplus \dotsb \oplus a_d \notin A_{d+1} \ominus b_1$.
Since 
\[ \frac{1}{1-\frac{c}{d^2}} K \le \frac{c}{d^2-c}n \le \frac{c}{(d-1)^2}n, \]
we can use induction to get a subgroup $H$ of $G$ and $x_2, \dotsc, x_d\in G$ such that for $j = 2, \dotsc, d$ we have
\[ |A_j \tri (H \oplus x_j)|, \left|(A_{d+1} \ominus b_1) \tri \left(H \oplus \bigoplus_{j=2}^d x_j \right) \right| \le C\prod_{i=1}^{d-1}\left(1+\frac{1}{i^2}\right)\frac{1}{1-\frac{c}{d^2}}K. \]

Since $|A_d \cap (H \oplus x_d)| \ge n - K - C\prod_{i=1}^{d-1}(1+\frac{1}{i^2})\frac{1}{1-\frac{c}{d^2}}K$,
we repeat the same pigeonhole argument on $A_d\cap (H \oplus x_d)$ to find a $b_d \in A_d \cap (H \oplus x_d)$ such that there are at most 
\begin{align*}
\frac{1}{n - K - C\prod_{i=1}^{d-1}\left(1+\frac{1}{i^2}\right)\frac{1}{1-\frac{c}{d^2}}K} Kn^{d-1} &\le \frac{1}{1-\frac{c}{d^2}-C\prod_{i=1}^{d-1}\left(1+\frac{1}{i^2}\right) \frac{c}{d^2-c}} Kn^{d-2}\\
&\le \frac{1}{1-C_1\frac{c}{d^2-c}}Kn^{d-2} \\
&\le \left(1 + \frac{C_2c}{d^2-c}\right)Kn^{d-2}\\
&\le \left(1 + \frac{1}{d^2}\right)Kn^{d-2}
\end{align*}
$(d-1)$-tuples $(a_1, \dotsc, a_{d-1}) \in A_1 \times \dotsb A_{d-1}$ with $a_1 \oplus \dotsb \oplus a_{d-1} \oplus b_d \notin A_{d+1}$,
for some absolute constants $C_1,C_2>0$ depending on $C$, by making $c$ sufficiently small.
Now $(1+\frac{1}{d^2})K \le cn/(d-1)^2$, so by induction again, there exist a subgroup $H'$ of $G$ and elements $x_1, x_2', \dotsc, x_{d-1}' \in G$ such that for $k = 2, \dotsc, d-1$ we have
\begin{equation*}
|A_1 \tri (H' \oplus x_1)|, |A_k \tri (H' \oplus x_k')|, \left|(A_{d+1} \ominus b_d) \tri \left(H' \oplus x_1 \oplus \bigoplus_{k=2}^{d-1} x_k' \right) \right|
 \le C\prod_{i=1}^{d-1}\left(1+\frac{1}{i^2}\right) \left(1 + \frac{1}{d^2}\right) K.
\end{equation*}
From this, it follows that $|(H \oplus x_k) \cap (H' \oplus x_k')| \ge n - K - 2C\prod_{i=1}^d (1 + \frac{1}{i^2}) K = n - O(K)$.
Since $(H \oplus x_k) \cap (H' \oplus x_k')$ is non-empty, it has to be a coset of $H' \cap H$.
If $H' \neq H$, then $|H' \cap H| \le n/2 + O(K)$, a contradiction since $c$ is sufficiently small.
Therefore, $H=H'$, and $H \oplus x_k = H' \oplus x_k'$.
So we have
\[ |A_i \tri (H \oplus x_i)|, \left|A_{d+1} \tri \left(H \oplus \bigoplus_{\l=1}^{d-1} x_\l \oplus b_d \right) \right| \le C\prod_{i=1}^d\left(1+\frac{1}{i^2}\right)K. \]
Since $b_d \in H \oplus x_d$, we also obtain
\[ \left|A_{d+1} \tri \left(H \oplus \bigoplus_{i=1}^d x_i \right) \right| \le C\prod_{i=1}^d\left(1+\frac{1}{i^2}\right)K. \qedhere \]
\end{proof}

To apply Lemma~\ref{cor:a5}, we first need to know that removing $K$ points from a set does not change the number of ordinary hyperplanes it spans by too much.

\begin{lemma}\label{lem:stability}
Let $P$ be a set of $n$ points in $\RR\PP^d$, $d \ge 2$, where every $d$ points span a hyperplane.
Let $P'$ be a subset that is obtained from $P$ by removing at most $K$ points.
If $P$ spans $m$ ordinary hyperplanes, then $P'$ spans at most $m + \frac{1}{d}K\binom{n-1}{d-1}$
ordinary hyperplanes.
\end{lemma}

\begin{proof}
Fix a point $p\in P$.
Since every $d$ points span a hyperplane, there are at most $\binom{n-1}{d-1}$ sets of $d$ points from $P$ containing $p$ that span a hyperplane through $p$.
Thus, the number of $(d+1)$-point hyperplanes through $p$ is at most $\frac{1}{d}\binom{n-1}{d-1}$, since a set of $d+1$ points that contains $p$ has $d$ subsets of size $d$ that contain $p$.
If we remove points of $P$ one-by-one to obtain $P'$, we thus create at most $\frac{1}{d}K\binom{n-1}{d-1}$ ordinary hyperplanes.
\end{proof}

The following lemma then translates the additive combinatorial Lemma~\ref{cor:a5} to our geometric setting.

\begin{lemma}\label{lem:curve}
Let $d \ge 4$, $K>0$, and suppose $n \ge C(d^3K + d^4)$ for some sufficiently large absolute constant $C>0$.
Let $P$ be a set of $n$ points in $\RR\PP^d$ where every $d$ points span a hyperplane.
Suppose $P$ spans at most $K\binom{n-1}{d-1}$ ordinary hyperplanes,
and all but at most $dK$ points of $P$ lie on an elliptic normal curve or a rational singular curve $\delta$.
Then $P$ differs in at most $O(dK+d^2)$ points from a coset $H \oplus x$ of a subgroup $H$ of $\delta^*$, the smooth points of $\delta$, for some $x$ such that $(d+1)x \in H$.
In particular, $\delta$ is either an elliptic normal curve or a rational acnodal curve.
\end{lemma}

\begin{proof}
Let $P' = P \cap \delta^*$.
Then by Lemma~\ref{lem:stability}, $P'$ spans at most $K\binom{n-1}{d-1}+d\frac{1}{d}K\binom{n-1}{d-1} = 2K\binom{n-1}{d-1}$ ordinary hyperplanes. 

First suppose $\delta$ is an elliptic normal curve or a rational cuspidal or acnodal curve.
If $a_1, \dotsc, a_d \in \delta^*$ are distinct, then by Propositions~\ref{prop:elliptic_group} and~\ref{prop:rational_group}, the hyperplane through $a_1, \dotsc, a_d$ meets $\delta$ again in the unique point $a_{d+1} = \ominus(a_1 \oplus \dotsb \oplus a_d)$. 
This implies that $a_{d+1} \in P'$ for all but at most $d!O(K\binom{n-1}{d-1})$ $d$-tuples $(a_1, \dotsc, a_d) \in (P')^d$ with all $a_i$ distinct.
There are also at most $\binom{d}{2}n^{d-1}$ $d$-tuples $(a_1, \dotsc, a_d) \in (P')^d$ for which the $a_i$ are not all distinct.
Thus, $a_1 \oplus \dotsb \oplus a_d \in \ominus P'$ for all but at most $O((dK+d^2)n^{d-1})$ $d$-tuples $(a_1, \dotsc, a_d) \in (P')^d$.
Applying Lemma~\ref{cor:a5} with $A_1 = \dotsb = A_d = P'$ and $A_{d+1} = \ominus P'$, we obtain a finite subgroup $H$ of $\delta^*$ and a coset $H \oplus x$ such that $|P' \tri (H \oplus x)| = O(dK+d^2)$ and 
$|\ominus P' \tri (H \oplus dx)| = O(dK+d^2)$, the latter being equivalent to $|P' \tri (H \ominus dx)| = O(dK+d^2)$. 
Thus we have $|(H \oplus x) \tri (H \ominus dx)| = O(dK+d^2)$, which implies $(d+1)x \in H$.
Also, $\delta$ cannot be cuspidal, otherwise by Proposition~\ref{prop:rational_group} we have $\delta^* \cong (\RR, +)$, which has no finite subgroup of order greater than~$1$.

Now suppose $\delta$ is a rational crunodal curve.
By Proposition~\ref{prop:rational_group}, there is a bijective map $\phi: (\RR,+) \times \ZZ_2 \rightarrow \delta^*$ such that $d+1$ points in $\delta^*$ lie in a hyperplane if and only if they sum to $h$, where $h=\phi(0,0)$ or $\phi(0,1)$ depending on the curve $\delta$.
If $h=\phi(0,0)$ then the above argument follows through, and we obtain a contradiction as we have by Proposition~\ref{prop:rational_group} that $\delta^* \cong (\RR,+) \times \ZZ_2$, which has no finite subgroup of order greater than~$2$.
Otherwise, the hyperplane through distinct $a_1, \dotsc, a_d \in \delta^*$ meets $\delta$ again in the unique point $a_{d+1} = \phi(0,1) \ominus(a_1 \oplus \dotsb \oplus a_d)$. 
As before, this implies that $a_{d+1} \in P'$ for all but at most $O((dK+d^2)n^{d-1})$ $d$-tuples $(a_1, \dotsc, a_d) \in (P')^d$,
or equivalently $a_1 \oplus \dotsb \oplus a_d \in \phi(0,1) \ominus P'$. 
Applying Lemma~\ref{cor:a5} with $A_1 = \dotsb = A_d = P'$ and $A_{d+1} = \phi(0,1) \ominus P'$, we obtain a finite subgroup $H$ of $\delta^*$, giving a contradiction as before.
\end{proof}

We can now prove Theorem~\ref{thm:main1}.

\begin{proof}[Proof of Theorem~\ref{thm:main1}]
By Lemma~\ref{lem:intermediate}, all but at most $O(d2^dK)$ points of $P$ are contained in a hyperplane or an irreducible curve $\delta$ of degree $d+1$ that is either elliptic or rational and singular. 
In the prior case, we get Case~\ref{case:hyperplane} of the theorem, so suppose we are in the latter case.
We then apply Lemma~\ref{lem:curve} to obtain Case~\ref{case:curve} of the theorem, completing the proof.
\end{proof}

\section{Extremal configurations}\label{sec:extremal}

We prove Theorems~\ref{thm:main2} and~\ref{thm:main3} in this section.
It will turn out that minimising the number of ordinary hyperplanes spanned by a set is equivalent to maximising the number of $(d+1)$-point planes,
thus we can apply Theorem~\ref{thm:main1} in both theorems.
Then we only have two cases to consider, where most of our point set is contained either in a hyperplane or a coset of a subgroup of an elliptic normal curve or the smooth points of a rational acnodal curve.

The first case is easy, and we get the following lower bound.

\begin{lemma}\label{lem:hyperplane}
Let $d \ge 4$, $K \ge 1$, and let $n \ge 2dK$.
Let $P$ be a set of $n$ points in $\RR\PP^d$ where every $d$ points span a hyperplane.
If all but $K$ points of $P$ lie on a hyperplane, then $P$ spans at least $\binom{n-1}{d-1}$ ordinary hyperplanes, with equality if and only if $K = 1$.
\end{lemma}

\begin{proof}
Let $\Pi$ be a hyperplane with $|P \cap \Pi| = n - K$.
Since $n-K>d$, any ordinary hyperplane spanned by $P$ must contain at least one point not in $\Pi$.
Let $m_i$ be the number of hyperplanes containing exactly $d-1$ points of $P\cap\Pi$ and exactly $i$ points of $P\setminus\Pi$, $i=1,\dots,K$.
Then the number of unordered $d$-tuples of elements from $P$ with exactly $d-1$ elements in $\Pi$ is
\[ K\binom{n-K}{d-1} = m_1 + 2m_2 + 3m_3 + \dots + Km_K.\]

Now consider the number of unordered $d$-tuples of elements from $P$ with exactly $d-2$ elements in $\Pi$, which equals $\binom{K}{2}\binom{n-K}{d-2}$.
One way to generate such a $d$-tuple is to take one of the $m_i$ hyperplanes containing $i$ points of $P\setminus\Pi$ and $d-1$ points of $P\cap\Pi$, choose two of the $i$ points, and remove one of the $d-1$ points.
Since any $d$ points span a hyperplane, there is no overcounting.
This gives
\begin{align*}
\binom{K}{2}\binom{n-K}{d-2} &\ge (d-1)\left(\binom{2}{2}m_2+\binom{3}{2}m_3+\binom{4}{2}m_4 + \dotsb\right)\\
&\ge \frac{d-1}{2} (2m_2+3m_3+4m_4+\dotsb).
\end{align*}
Hence the number of ordinary hyperplanes is at least
\[ m_1\ge K\binom{n-K}{d-1}-\frac{K(K-1)}{d-1}\binom{n-K}{d-2} = K\binom{n-K}{d-1}\frac{n-2K-d+3}{n-K-d+2}.\]
We next show that for all $K\ge 2$, if $n \ge 2dK$ then
\[ K\binom{n-K}{d-1}\frac{n-2K-d+3}{n-K-d+2} > \binom{n-1}{d-1}.\]
This is equivalent to
\begin{equation}\label{ineq1}
K > \frac{n-K+1}{n-2K-d+3}\prod_{i=1}^{K-2}\frac{n-i}{n-d-i+1}.
\end{equation}
Note that
\begin{equation}\label{ineq2}
\frac{n-K+1}{n-2K-d+3} < 2
\end{equation}
if $n > 3K+2d-5$ and
\begin{equation}\label{ineq3}
\frac{n-i}{n-d-i+1} <\frac{i+2}{i+1}
\end{equation}
if $n\ge (i+2)d$ for each $i=1,\dots,K-2$.
However, since $2dK > (i+2)d$ and also $2dK > 4K+2d-5$, the inequality \eqref{ineq1} now follows from \eqref{ineq2} and \eqref{ineq3}.
\end{proof}

The second case needs more work.
We first consider the number of ordinary hyperplanes spanned by a coset of a subgroup of the smooth points $\delta^*$ of an elliptic normal curve or a rational acnodal curve.
By Propositions~\ref{prop:elliptic_group} and~\ref{prop:rational_group}, we can consider $\delta^*$  as a group isomorphic to either $\RR/\ZZ$ or $\RR/\ZZ \times \ZZ_2$.
Let $H \oplus x$ be a coset of a subgroup $H$ of $\delta^*$ of order $n$ where $(d+1)x = \ominus c \in H$.
Since $H$ is a subgroup of order $n$ of $\RR/\ZZ$ or $\RR/\ZZ \times \ZZ_2$, we have that either $H$ is cyclic, or $\ZZ_{n/2}\times\ZZ_2$ when $n$ is divisible by $4$.
The exact group will matter only when we make exact calculations.

Note that it follows from the group property that any $d$ points on $\delta^*$ span a hyperplane.
Also, since any hyperplane intersects $\delta^*$ in $d+1$ points, counting multiplicity, it follows that an ordinary hyperplane of $H\oplus x$ intersects $\delta^*$ in $d$ points, of which exactly one of them has multiplicity $2$, and the others multiplicity $1$.
Denote the number of ordered $k$-tuples $(a_1, \dotsc, a_k)$ with distinct $a_i \in H$ that satisfy $m_1 a_1 \oplus \dotsb \oplus m_k a_k = c$ by $[m_1, \dotsc, m_k; c]$.
Then the number of ordinary hyperplanes spanned by $H \oplus x$ is
\begin{equation}\label{expr1} \frac{1}{(d-1)!} [2, \!\!\underbrace{1, \dotsc, 1}_\text{$d-1$ times}\!; c].
\end{equation}
We show that we can always find a value of $c$ for which \eqref{expr1} is at most $\binom{n-1}{d-1}$.

\begin{lemma}\label{lem:coset1}
Let $\delta^*$ be an elliptic normal curve or the smooth points of a rational acnodal curve in $\RR\PP^d$, $d \ge 2$.
Then any finite subgroup $H$ of $\delta^*$ of order $n$ has a coset $H \oplus x$ with $(d+1)x \in H$, that spans at most $\binom{n-1}{d-1}$ ordinary hyperplanes.
Furthermore, if $d+1$ and $n$ are coprime, then any such coset spans exactly $\binom{n-1}{d-1}$ ordinary hyperplanes.
\end{lemma}

\begin{proof}
It suffices to show that there exists $c \in H$ such that the number of solutions $(a_1, \dotsc, a_d)\in H^d$ of the equation $2a_1 \oplus a_2 \oplus \dotsb \oplus a_d = c$, where $c = \ominus(d+1)x$, is at most $(d-1)!\binom{n-1}{d-1}$.

Fix $a_1$ and consider the substitution $b_i = a_i - a_1$ for $i = 2, \dotsc, d$.
Note that $2a_1 \oplus \dotsb \oplus a_d = c$ and $a_1,\dots,a_d$ are distinct if and only if $b_2 \oplus \dotsb \oplus b_d = c \ominus (d+1)a_1$ and $b_2,\dots,b_d$ are distinct and non-zero.
Let
\[ A_{c,j} = \setbuilder{(j, a_2, \dotsc, a_d)}{2j \oplus a_2 \oplus \dotsb \oplus a_d = c, \text{$a_2, \dotsc, a_d \in H \setminus \{j\}$ distinct}}, \]
and let
\[ B_k = \setbuilder{(b_2, \dotsc, b_d)}{b_2 \oplus \dotsb \oplus b_d = k, \text{$b_2, \dotsc, b_d \in H \setminus \{0\}$ distinct}}. \]
Then $|A_{c,j}| = |B_{c \ominus (d+1)j}|$, and the number of ordinary hyperplanes spanned by $H \oplus x$ is
\[ \frac{1}{(d-1)!} \sum_{j\in H} |A_{c,j}|. \]

If $d+1$ is coprime to $n$, then $c \ominus (d+1)j$ runs through all elements of $H$ as $j$ varies.
So we have $\sum_j |B_{c \ominus (d+1)j}| = (n-1)\dotsb (n-d+1)$, hence for all $c$,
\[ \frac{1}{(d-1)!} \sum_{j\in H} |A_{c,j}| = \binom{n-1}{d-1}. \]

If $d+1$ is not coprime to $n$, then $c \ominus (d+1)j$ runs through a coset of a subgroup of $H$ of size $n/\gcd(d+1,n)$ as $j$ varies.
We now have
\[ \sum_{j \in H} |B_{c \ominus (d+1)j}| = \gcd(d+1,n) \sum_{k \in c \ominus (d+1)H} |B_k|. \]
Summing over $c$ gives
\begin{align*}
\sum_{c \in H} \sum_{j \in H} |A_{c,j}| &= \gcd(d+1,n) \sum_{c \in H} \sum_{k \in c \ominus (d+1)H} |B_k|\\
&= \gcd(d+1,n) \frac{n}{\gcd(d+1,n)} (n-1)\dotsb (n-d+1)\\
& = n (n-1)\dotsb (n-d+1).
\end{align*}
By the pigeonhole principle, there must then exist a $c$ such that
\[ \frac{1}{(d-1)!} \sum_{j\in H} |A_{c,j}| \le \binom{n-1}{d-1}. \qedhere \]
\end{proof}

We next want to show that $[2, \!\!\overbrace{1, \dotsc, 1}^\text{$d-1$ times}\!\!; c]$ is always very close to $(d-1)!\binom{n-1}{d-1}$, independent of $c$ or the group $H$.
Before that, we prove two simple properties of $[m_1, \dotsc, m_k; c]$.

\begin{lemma}\label{lem:upper}
$[m_1,\dots,m_k;c]\le 2m_k(k-1)!\binom{n}{k-1}$.
\end{lemma}

\begin{proof}
Consider a solution $(a_1,\dotsc,a_k)$ of $m_1a_1 \oplus \dotsb \oplus m_ka_k=c$ where all the $a_i$ are distinct.
We can choose $a_1,\dotsc,a_{k-1}$ arbitrarily in $(k-1)!\binom{n}{k-1}$ ways, and $a_k$ satisfies the equation $m_ka_k = c \ominus m_1a_1 \ominus \dotsb \ominus m_{k-1}a_{k-1}$, which has at most $m_k$ solutions if $H=\ZZ_n$ and at most $2m_k$ solutions if $H=\ZZ_2\times\ZZ_{n/2}$.
\end{proof}

\begin{lemma}\label{lem:recurrence}
We have the recurrence relation
\begin{align}
[m_1,\dots,m_{k-1},1;c] = (k-1)!\binom{n}{k-1} &- [m_1 +1,m_2,\dots,m_{k-1};c]\notag\\
&- [m_1,m_2 +1,m_3,\dots,m_{k-1};c]\notag\\
&- \dotsb\notag\\ 
&- [m_1,\dots,m_{k-2},m_{k-1}+1;c]. \notag
\end{align}
\end{lemma}

\begin{proof}
We can arbitrarily choose distinct values from $H$ for $a_1,\dots,a_{k-1}$, which determines $a_k$, and then we have to subtract the number of $k$-tuples where $a_k$ is equal to one of the other $a_i$, $i=1,\dots,k-1$.
\end{proof}

\begin{lemma}\label{lem:2111}
\[ [2, \!\!\underbrace{1, \dotsc, 1}_\text{$d-1$ times}\!\!; c] =  (d-1)! \left( \binom{n-1}{d-1} + \eps(d,n)\right),\]
where
\[ |\eps(d,n)| = \begin{cases}
 O\left(2^{-d/2}\binom{n}{(d-1)/2}+\binom{n}{(d-3)/2}\right) & \text{if $d$ is odd,}\\
 O\left(d 2^{-d/2}\binom{n}{d/2-1}+\binom{n}{d/2-2}\right) & \text{if $d$ is even.}
\end{cases} \] 
\end{lemma}

\begin{proof}
Applying Lemma~\ref{lem:recurrence} once, we obtain
\[ [2, \!\!\underbrace{1, \dotsc, 1}_\text{$d-1$ times}\!\!; c] = (d-1)!\binom{n}{d-1} - [3, \!\!\underbrace{1, \dotsc, 1}_\text{$d-2$ times}\!\!; c] - (d-2)[2, 2, \!\!\underbrace{1, \dotsc, 1}_\text{$d-3$ times}\!\!; c]. \]
Note that at each stage of the recurrence in Lemma~\ref{lem:recurrence} (as long as it applies),
there are $(d-1)(d-2)\dotsb(d-k)$ terms of length $d-k$, where we define the \emph{length} of $[m_1, \dotsc, m_k; c]$ to be $k$.

If $d$ is odd, we can continue this recurrence until we reach
\begin{align*}
[2, \!\!\underbrace{1, \dotsc, 1}_\text{$d-1$ times}\!\!; c] &= (d-1)! \left( \binom{n}{d-1} - \binom{n}{d-2} + \dotsb + (-1)^{(d+1)/2} \binom{n}{(d+1)/2} \right)\\
&\qquad + (-1)^{(d-1)/2}R,
\end{align*}
where $R$ is the sum of $(d-1)(d-2)\dotsb(d-(d-1)/2)$ terms of length $(d+1)/2$.
Among these there are \[\frac{\binom{d-1}{2}\binom{d-3}{2}\dotsb\binom{2}{2}}{(\frac{d-1}{2})!} = (d-2)(d-4)\dotsb 3 \cdot1\] terms of the form $[2, \dotsc, 2; c]$.
We now write $R=A+B$, where $A$ is the same sum as $R$, except that we replace each occurrence of $[2,\dots,2;c]$ by $[1,\dots,1;c]$, and
\[ B := (d-2)(d-4)\dotsb 3\cdot 1 ([\underbrace{2,\dotsc,2}_\text{$\frac{d+1}{2}$ times};c] - [\!\underbrace{1,\dotsc,1}_\text{$\frac{d+1}{2}$ times}\!;c]).\]
We next bound $A$ and $B$.
We apply Lemma~\ref{lem:recurrence} to each term in $A$, after which we obtain $(d-1)(d-2)\dotsb(d-(d+1)/2)$ terms of length $(d-1)/2$.
Then using the bound in Lemma~\ref{lem:upper}, we obtain 
\begin{align*}
A &= (d-1)!\binom{n}{(d-1)/2} - O\left((d-1)(d-2)\dotsb(d-(d+1)/2)\left(\tfrac{d-3}{2}\right)!\binom{n}{(d-3)/2}\right)\\
&= (d-1)!\left(\binom{n}{(d-1)/2} - O\left(\binom{n}{(d-3)/2}\right) \right).
\end{align*}
For $B$, we again use Lemma~\ref{lem:upper} to get
\begin{align*}
|B| &= O\left((d-2)(d-4)\dotsb 3\cdot 1 \left(\frac{d-1}{2}\right)! \binom{n}{(d-1)/2} \right)\\
&= O\left((d-2)(d-4)\dotsb 3\cdot 1 \cdot 2^{-\frac{d-1}{2}}(d-1)(d-3)\dotsb 4 \cdot 2 \binom{n}{(d-1)/2} \right)\\
&= O\left((d-1)!2^{-\frac{d-1}{2}}\binom{n}{(d-1)/2}\right).
\end{align*}

Thus we obtain
\begin{multline*}
[2, \!\!\underbrace{1, \dotsc, 1}_\text{$d-1$ times}\!\!; c] = (d-1)! \left( \binom{n}{d-1} - \binom{n}{d-2} + \dotsb +(-1)^{\frac{d+1}{2}} \binom{n}{(d+1)/2} \right)\\
 + (-1)^{\frac{d-1}{2}}(d-1)!\left(\binom{n}{(d-1)/2} - O\left(\binom{n}{(d-3)/2}\right) \right) + (-1)^{\frac{d-1}{2}}B\\
= (d-1)!\left( \binom{n-1}{d-1} + (-1)^{\frac{d+1}{2}} O\left(\binom{n}{(d-3)/2}\right)\pm O\left(2^{-\frac{d-1}{2}}\binom{n}{(d-1)/2}\right)\right),
\end{multline*}
which finishes the proof for odd $d$.

If $d$ is even, we obtain
%\begin{align*}
%[2, \!\!\underbrace{1, \dotsc, 1}_\text{$d-1$ times}\!\!; c] &= (d-1)! \left( \binom{n}{d-1} - \binom{n}{d-2} + \dotsb + (-1)^{\frac{d}{2}+1} \binom{n}{d/2} \right)\\
%&\quad + (-1)^{d/2}R,
%\end{align*}
\begin{equation*}
[2, \!\!\underbrace{1, \dotsc, 1}_\text{$d-1$ times}\!\!; c] = (d-1)! \left( \binom{n}{d-1} - \binom{n}{d-2} + \dotsb + (-1)^{\frac{d}{2}+1} \binom{n}{d/2} \right)
 + (-1)^{d/2}R,
\end{equation*}
where $R$ now is the sum of $(d-1)(d-2)\dotsb(d-d/2)$ terms of length $d/2$.
Among these there are
\[ \frac{(d-1)\binom{d-2}{2}\binom{d-4}{2}\dotsb\binom{2}{2}}{(\frac{d-2}{2})!} + \frac{2\binom{d-1}{3}\binom{d-4}{2}\dotsb\binom{2}{2}}{(\frac{d-4}{2})!} = (d+1)(d-1)\dotsb7\cdot5\]
terms of the form $[3,2,\dots,2;c]$.
Again we write $R=A+B$, where $A$ is the same sum as $R$, except that each occurrence of $[3,2,\dots,2;c]$ is replaced by $[1,\dots,1;c]$, and
\[ B := (d+1)(d-1)\dotsb7\cdot5([3,\!\!\underbrace{2,\dotsc,2}_\text{$\frac{d}{2}-1$ times}\!\!;c] - [\underbrace{1,\dotsc,1}_\text{$\frac{d}{2}$ times};c]).\]
Similar to the previous case, we obtain
\[ A = (d-1)!\left(\binom{n}{d/2-1}-O\left(\binom{n}{d/2-2}\right)\right)\]
and \[ |B| = O\left((d+1)(d-1)\dotsb7\cdot5(\tfrac{d}{2}-1)!\binom{n}{d/2-1}\right) = O\left(2^{-d/2}d!\binom{n}{d/2-1}\right),\]
which finishes the proof for even $d$.
\end{proof}

Computing $[2, \dotsc, 2; c]$ and $[3, 2, \dotsc, 2; c]$ exactly is more subtle and depends on $c$ and the group $H$.
We do not need this for the asymptotic Theorems~\ref{thm:main2} and~\ref{thm:main3}, and will only need to do so when computing exact extremal values.

To show that a coset is indeed extremal, we first consider the effect of adding a single point.
The case where the point is on the curve is done in Lemma~\ref{lem:7.7-}, while Lemma~\ref{lem:7.7} covers the case where the point is off the curve.
We then obtain a more general lower bound in Lemma~\ref{lem:coset2}.

\begin{lemma}\label{lem:7.7-}
Let $\delta^*$ be an elliptic normal curve or the smooth points of a rational acnodal curve in $\RR\PP^d$, $d \ge 2$.
Suppose $H \oplus x$ is a coset of a finite subgroup $H$ of $\delta^*$ of order $n$, with $(d+1)x \in H$.
Let $p\in\delta^*\setminus (H\oplus x)$.
Then there are at least $\binom{n}{d-1}$ hyperplanes through $p$ that meet $H \oplus x$ in exactly $d-1$ points.
\end{lemma}

\begin{proof}
Take any $d-1$ points $p_1, \dotsc, p_{d-1} \in H \oplus x$.
Suppose that the (unique) hyperplane through $p,p_1,\dots,p_{d-1}$ contains another point $p'\in H\oplus x$.
Since $p\oplus p_1\oplus \dots \oplus p_{d-1}\oplus p' = 0$ by Propositions~\ref{prop:elliptic_group} and~\ref{prop:rational_group}, we obtain that $p\in H\ominus dx$.
Since $(d+1)x\in H$, we obtain $p\in H\oplus x$, a contradiction.
Therefore, the hyperplane through $p,p_1,\dots,p_{d-1}$ does not contain any other point of $H\oplus x$.

It remains to show that if $\{p_1,\dots,p_{d-1}\}\neq\{p_1',\dots,p_{d-1}'\}$ where also $p_1',\dots,p_{d-1}'\in H\oplus x$, then the two sets span different hyperplanes with $p$.
Suppose they span the same hyperplane.
Then $\ominus(p \oplus p_1 \oplus \dotsb \oplus p_{d-1})$ also lies on this hyperplane, but not in $H\oplus x$, as shown above.
Also, $p_i'\notin\{p_1,\dots,p_{d-1}\}$ for some $i$, and then $p_1,\dots,p_{d-1},p_i'$, and $\ominus(p \oplus p_1 \oplus \dotsb \oplus p_{d-1})$ are $d+1$ distinct points on a hyperplane, so their sum is $0$, which implies $p=p_i'$, a contradiction.

So there are $\binom{n}{d-1}$ hyperplanes through $p$ meeting $H \oplus x$ in exactly $d-1$ points.
\end{proof}

The following Lemma generalises \cite{GT13}*{Lemma~7.7}, which states that if $\delta^*$ is an elliptic curve or the smooth points of an acnodal cubic curve in the plane, $H\oplus x$ is a coset of a finite subgroup of order $n>10^4$, and if $p\notin\delta^*$, then there are at least $n/1000$ lines through $p$ that pass through exactly one element of $H\oplus x$.
A naive generalisation to dimension $3$ would state that if $\delta^*$ is an elliptic or acnodal space quartic curve with a finite subgroup $H$ of sufficiently large order $n$, and $x\in\delta^*$ and $p\notin\delta^*$, then there are $\Omega(n^2)$ planes through $p$ and exactly two elements of $H\oplus x$.
This statement is false, even if we assume that $4x\in H$ (the analogous assumption $3x\in H$ is not made in \cite{GT13}), as can be seen from the following example.

Let $\delta$ be an elliptic quartic curve obtained from the intersection of a circular cylinder in $\RR^3$ with a sphere which has centre $c$ on the axis $\ell$ of the cylinder.
Then $\delta$ is symmetric in the plane through $c$ perpendicular to $\ell$, and we can find a finite subgroup $H$ of any even order $n$ such that the line through any element of $H$ parallel to $\ell$ intersects $H$ in two points.
If we now choose $p$ to be the point at infinity on $\ell$, then we obtain that any plane spanned by $p$ and two points of $H$ not collinear with $p$, intersects $H$ in two more points.
Note that the projection $\pi_p$ maps $\delta$ to a conic, so is not generically one-to-one.
The number of such $p$ is bounded by the trisecant lemma (Lemma~\ref{lem:projection2}).
However, as the next lemma shows, a generalisation of \cite{GT13}*{Lemma~7.7} holds except that in dimension~3 we have to exclude such points $p$.
\begin{lemma}\label{lem:7.7}
Let $\delta$ be an elliptic normal curve or a rational acnodal curve in $\RR\PP^d$, $d \ge 2$, and let $\delta^*$ be its set of smooth points.
Let $H$ be a finite subgroup of $\delta^*$ of order $n$, where $n\ge Cd^4$ for some sufficiently large absolute constant $C>0$.
Let $x\in\delta^*$ satisfy $(d+1)x\in H$.
Let $p\in\RR\PP^d\setminus\delta^*$.
If $d=3$, assume furthermore that $\delta$ is not contained in a quadric cone with vertex $p$.
Then there are at least $c\binom{n}{d-1}$ hyperplanes through $p$ that meet the coset $H \oplus x$ in exactly $d-1$ points, for some sufficiently small absolute constant $c>0$.
\end{lemma}

\begin{proof}
We prove by induction on $d$ that under the given hypotheses there are at least $c'\prod_{i=2}^d(1-\frac{1}{i^2})\binom{n}{d-1}$ such hyperplanes for some sufficiently small absolute constant $c'>0$.
The base case $d = 2$ is given by \cite{GT13}*{Lemma~7.7}.

Next assume that $d\ge 3$, and that the statement holds for $d-1$.
Fix a $q\in H\oplus x$, and consider the projection $\pi_q$.
Since $q$ is a smooth point of $\delta$, $\overline{\pi_q(\delta\setminus\{q\})}$ is a non-degenerate curve of degree $d$ in $\RR\PP^{d-1}$ (otherwise its degree would be at most $d/2$, but a non-degenerate curve has degree at least $d-1$).
The projection $\pi_q$ can be naturally extended to have a value at $q$, by setting $\pi_q(q)$ to be the point where the tangent line of $\delta$ at $q$ intersects the hyperplane onto which $\delta$ is projected.
(This point is the single point in $\overline{\pi_q(\delta\setminus\{q\})}\setminus\pi_q(\delta\setminus\{q\})$.)
The curve $\pi_q(\delta)$ has degree $d$ and is either elliptic or rational and acnodal, hence it has a group operation $\boxplus$ such that $d$ points are on a hyperplane in $\RR\PP^{d-1}$ if and only if they sum to the identity.

Observe that any $d$ points $\pi_q(p_1),\dots,\pi_q(p_d)\in\pi_q(\delta^*)$ lie on a hyperplane in $\RR\PP^{d-1}$ if and only if $p_1\oplus\dots\oplus p_d \oplus q = 0$.
By Proposition~\ref{prop:unique} it follows that the group on $\pi_q(\delta^*)$ obtained by transferring the group $(\delta^*,\oplus)$ by $\pi_q$ is a translation of $(\pi_q(\delta^*),\boxplus)$.
In particular, $\pi_q(H\oplus x) = H'\boxplus x'$ for some subgroup $H'$ of $(\pi_q(\delta^*),\boxplus)$ of order $n$, and $(d+1)x'\in H'$.

We would like to apply the induction hypothesis, but we can only do that if $\pi_q(p)\notin\pi_q(\delta^*)$, and when $d=4$, if $\pi_q(p)$ is not the vertex of a quadric cone containing $\pi_q(\delta)$.
We next show that there are only $O(d^2)$ exceptional points $q$ to which we cannot apply induction.

Note that $\pi_q(p)\in\pi_q(\delta^*)$ if and only if the line $pq$ intersects $\delta$ with multiplicity $2$, which means we have to bound the number of these lines through $p$.
To this end, we consider the projection of $\delta$ from the point $p$.
Suppose that $\pi_p$ does not project $\delta$ generically one-to-one to a degree $d+1$ curve in $\RR\PP^{d-1}$.
Then $\pi_p(\delta)$ has degree at most $(d+1)/2$.
However, its degree is at least $d-1$ because it is non-degenerate.
It follows that $d=3$, and that $\pi_p(\delta)$ has degree $2$ and is irreducible, so $\delta$ is contained in a quadric cone with vertex $p$, which we ruled out by assumption.

Therefore, $\pi_p$ projects $\delta$ generically one-to-one onto the curve $\pi_p(\delta)$, which has degree $d+1$ and has at most $\binom{d}{2}$ double points (this follows from the Pl\"ucker formulas after projecting to the plane \cite{W78}*{Chapter~III, Theorem~4.4}).
We thus have that an arbitrary point $p \in \RR\PP^d \setminus \delta$ lies on at most $O(d^2)$ secants or tangents of $\delta$ (or lines through two points of $\delta^*$ if $p$ is the acnode of $\delta$).

If $d=4$, we also have to avoid $q$ such that $\pi_q(p)$ is the vertex of a cone on which $\pi_q(\delta)$ lies.
Such $q$ have the property that if we first project $\delta$ from $q$ and then $\pi_q(\delta)$ from $\pi_q(p)$, then the composition of these two projections is not generically one-to-one.
Another way to do these to successive projections is to first project $\delta$ from $p$ and then $\pi_p(\delta)$ from $\pi_p(q)$.
Thus, we have that $\pi_p(q)$ is a point on the quintic $\pi_p(\delta)$ in $\RR\PP^3$ such that the projection of $\pi_p(\delta)$ from $\pi_p(q)$ onto $\RR\PP^2$ is not generically one-to-one.
However, there are only $O(1)$ such points by Lemma~\ref{lem:projection2}.
Thus there are at most $Cd^2$ points $q\in H\oplus x$ to which we cannot apply the induction hypothesis.

For all remaining $q\in H\oplus x$, we obtain by the induction hypothesis that there are at least $c'\prod_{i=2}^{d-1}(1-\frac{1}{i^2})\binom{n}{d-2}$ hyperplanes $\Pi$ in $\RR\PP^{d-1}$ through $\pi_q(p)$ and exactly $d-2$ points of $H'\boxplus x'$.
If none of these $d-2$ points equal $\pi_q(q)$, then $\pi_q^{-1}(\Pi)$ is a hyperplane in $\RR\PP^d$ through $p$ and $d-1$ points of $H\oplus x$, one of which is $q$.
There are at most $\binom{n-1}{d-3}$ such hyperplanes in $\RR\PP^{d-1}$ through $\pi_q(q)$.
Therefore, there are at least $c'\prod_{i=2}^{d-1}(1-\frac{1}{i^2})\binom{n}{d-2} - \binom{n-1}{d-3}$ hyperplanes in $\RR\PP^d$ that pass through $p$ and exactly $d-1$ points of $H\oplus x$, one of them being $q$.
If we sum over all $n-Cd^2$ points $q$, we count each hyperplane $d-1$ times, and we obtain that the total number of such hyperplanes is at least
\begin{equation}\label{eq1}
\frac{n-Cd^2}{d-1}\left(c'\prod_{i=2}^{d-1}\left(1-\frac{1}{i^2}\right)\binom{n}{d-2} - \binom{n-1}{d-3}\right).
\end{equation}
It can easily be checked that 
\begin{equation}\label{eq2}
\frac{n-Cd^2}{d-1}\binom{n}{d-2} \ge \left(1 - \frac{1}{2d^2}\right)\binom{n}{d-1}
\end{equation}
if $n >2Cd^4$, and that
\begin{equation}\label{eq3}
c'\prod_{i=2}^{d-1}\left(1-\frac{1}{i^2}\right)\frac{1}{2d^2}\binom{n}{d-1} \ge  \frac{n-Cd^2}{d-1}\binom{n-1}{d-3}
\end{equation}
if $n > 4d^3/c'$.
It now follows from \eqref{eq2} and \eqref{eq3} that the expression \eqref{eq1} is at least 
\[ c'\prod_{i=2}^{d}\left(1-\frac{1}{i^2}\right)\binom{n}{d-1},\]
which finishes the induction.
\end{proof}

\begin{lemma}\label{lem:coset2}
Let $\delta^*$ be an elliptic normal curve or the smooth points of a rational acnodal curve in $\RR\PP^d$, $d \ge 4$, and let $H \oplus x$ be a coset of a finite subgroup $H$ of $\delta^*$, with $(d+1)x\in H$.
Let $A\subseteq H\oplus x$ and $B\subset\RR\PP^d\setminus(H\oplus x)$ with $|A|=a$ and $|B|=b$.
Let $P = (H \oplus x \setminus A) \cup B$ with $|P|=n$ be such that every $d$ points of $P$ span a hyperplane.
If $A$ and $B$ are not both empty and $n\ge C(a+b+d^2)d$ for some sufficiently large absolute constant $C>0$, then $P$ spans at least $(1+c)\binom{n-1}{d-1}$ ordinary hyperplanes for some sufficiently small absolute constant $c>0$.
\end{lemma}

\begin{proof}
We first bound from below the number of ordinary hyperplanes of $(H\oplus x)\setminus A$ that do not pass through a point of $B$.

The number of ordinary hyperplanes of $(H\oplus x)\setminus A$ that are disjoint from $A$ is
\[ \frac{1}{(d-1)!}\left|\setbuilder{(a_1,\dots,a_d)\in(H\setminus(A\ominus x))^d}{\begin{array}{c}2a_1\oplus a_2 \oplus\dotsb \oplus a_d=\ominus (d+1)x,\\ \text{$a_1,\dots,a_d$ are distinct}\end{array}}\right|.\]
 If we denote by by $[m_1, \dotsc, m_k]'$ the number of ordered $k$-tuples $(a_1, \dotsc, a_k)$ with distinct $a_i \in H\setminus(A\ominus x)$ that satisfy $m_1 a_1 \oplus \dotsb \oplus m_k a_k = \ominus (d+1)x$, then we obtain, similar to the proofs of Lemmas~\ref{lem:upper} and \ref{lem:recurrence}, that
\begin{align*}
[2, \!\!\underbrace{1, \dotsc, 1}_\text{$d-1$ times}\!]' &= (d-1)!\binom{n-b}{d-1} - [3, \!\!\underbrace{1, \dotsc, 1}_\text{$d-2$ times}\!]' - (d-2)[2, 2, \!\!\underbrace{1, \dotsc, 1}_\text{$d-3$ times}\!]'\\
&\ge  (d-1)!\binom{n-b}{d-1} - 2(d-2)!\binom{n-b}{d-2}-2(d-2)(d-2)!\binom{n-b}{d-2}\\
&= (d-1)!\binom{n-b}{d-1} - 2(d-1)!\binom{n-b}{d-2},
\end{align*}
and it follows that the number of ordinary hyperplanes of $(H\oplus x)\setminus A$ disjoint from $A$ is at least $\binom{n-b}{d-1}-2\binom{n-b}{2}$.

Next, we obtain an upper bound on the number of these hyperplanes that pass through a point $q\in B$.
Let the ordinary hyperplane $\Pi$ pass through $p_1,p_2,\dots,p_d\in (H\oplus x)\setminus A$,
with $p_1$ being the double point.
Since $q\in\Pi$ and any $d$ points determine a hyperplane, $\Pi$ is still spanned by $q,p_1,\dots,p_{d-1}$, after a relabelling of $p_2,\dots,p_d$.
Let $S$ be a minimal subset of $\{p_2,\dots,p_{d-1}\}$ such that the tangent line $\ell$ of $\delta$ at $p_1$ lies in the flat spanned by $S\cup\{q,p_1\}$.

If $S$ is empty, then $\ell$ is a tangent from $q$ to $\delta$, of which there are at most $d(d+1)$ (this follows again from projection and the Pl\"ucker formulas \citelist{\cite{W78}*{Chapter~IV, p.~117}\cite{NZ}*{Corollary~2.5}}).
Therefore, the number of ordinary hyperplanes through $p_1,p_2,\dots,p_d\in (H\oplus x)\setminus A$ with the tangent of $\delta$ at $p_1$ passing through $q$ is at most $d(d+1)\binom{n-b}{d-2}$.

If on the other hand $S$ is non-empty, then there is some $p_i$, say $p_{d-1}$, such that $q,p_1,\dots,p_{d-2}$ together with $\ell$ generate $\Pi$.
Therefore, $\Pi$ is determined by $p_1$, the tangent through $p_1$, and some $d-3$ more points $p_i$.
There are at most $(n-b)\binom{n-b-1}{d-3} = (d-2)\binom{n-b}{d-2}$ ordinary hyperplanes through $q$ in this case.

The number of ordinary hyperplanes of $(H\oplus x)\setminus A$ that contain a point from $A$ is at least
\[a\left(\binom{n-b}{d-1}-a\binom{n-b}{d-2}-(n-b)\binom{n-b-1}{d-3}\right) = a\binom{n-b}{d-1}-(a^2+a(d-2))\binom{n-b}{d-2},\]
since we can find such a hyperplane by choosing a point $p\in A$ and $d-1$ points $p_1,\dots,p_{d-1}\in(H\oplus x)\setminus A$, and then the remaining point $\ominus(p\oplus p_1\oplus\dots\oplus p_{d-1})$ might not be a new point in $(H\oplus x)\setminus A$ by either being in $A$ (possibly equal to $p$) or being equal to one of the $p_i$.
The number of these hyperplanes that also pass through some point of $B$ is at most $ab\binom{n-b}{d-2}$.

Therefore, the number of ordinary hyperplanes of $(H\oplus x)\setminus A$ that miss $B$ is at least
\begin{equation}\label{lower1}
(1+a)\binom{n-b}{d-1}-\left(2 + b(d(d+1)+d-2) + a^2+a(d-2) + ab\right)\binom{n-b}{d-2}.
\end{equation}

Next, assuming that $B\neq\emptyset$, we find a lower bound to the number of ordinary hyperplanes through exactly one point of $B$ and exactly $d-1$ points of $(H\oplus x)\setminus A$.
The number of hyperplanes through at least one point of $B$ and exactly $d-1$ points of $(H\oplus x)\setminus A$ is at least $bc'\binom{n-b}{d-1}-ab\binom{n-b}{d-2}$ by Lemmas~\ref{lem:7.7-} and~\ref{lem:7.7} for some sufficiently small absolute constant $c'>0$.
The number of hyperplanes through at least two points of $B$ and exactly $d-1$ points of $(H\oplus x)\setminus A$ is at most $\binom{b}{2}\binom{n-b}{d-2}$.
It follows that there are at least $bc'\binom{n-b}{d-1}-\bigl(ab+\binom{b}{2}\bigr)\binom{n-b}{d-2}$ ordinary hyperplanes passing though a point of~$B$.

Combining this with \eqref{lower1}, $P$ spans at least
\begin{equation*} (1+a+bc')\binom{n-b}{d-1}-\left(2 + b(d(d+1)+d-2) + a^2+a(d-2) + 2ab + \binom{b}{2}\right)\binom{n-b}{d-2} =: f(a,b)
\end{equation*}
ordinary hyperplanes.
Since 
\[f(a+1,b) - f(a,b) = \binom{n-b}{d-1}-(2a+2b+d-1)\binom{n-b}{d-2}\] 
is easily seen to be positive for all $a\ge 0$ as long as $n > (2a+2b+d-1)(d-1)+b+d-2$, we have without loss of generality that $a=0$ in the case that $b\ge 1$.
Then 
$f(0,b+1)-f(0,b)$ is easily seen to be at least 
\[ c'\binom{n-b-1}{d-1} - (d^2+d-2+b)\binom{n-b-1}{d-2},\]
which is positive for all $b\ge 1$ if $n \ge C(b+d^2)d$ for $C$ sufficiently large.
Also, $f(0,1) = (1+c')\binom{n-1}{d-1} - (d^2+2d)\binom{n-1}{d-2}) \ge (1+c)\binom{n-1}{d-1}$ if $n\ge Cd^3$.
This completes the proof in the case where $B$ is non-empty.

If $B$ is empty, then we can bound the number of ordinary hyperplanes from below by setting $b=0$ in \eqref{lower1}, and checking that the resulting expression
\[ (1+a)\binom{n}{d-1} - \left(d+a^2+a(d-2)\right)\binom{n}{d-2}\] is increasing in $a$ if $n > (2a+d-1)(d-1)+d-2$, and larger than $\frac32\binom{n-1}{d-1}$ if $n > Cd^3$.
\end{proof}

We are now ready to prove Theorems~\ref{thm:main2} and~\ref{thm:main3}.

\begin{proof}[Proof of Theorem~\ref{thm:main2}]
Let $P$ be the set of $n$ points.
By Lemma~\ref{lem:coset1}, we may assume that $P$ has at most $\binom{n-1}{d-1}$ ordinary hyperplanes.
Since $n \ge C d^3 2^d$, we may apply Theorem~\ref{thm:main1} to obtain that up to $O(d2^d)$ points, $P$ lies in a hyperplane or is a coset of a subgroup of an elliptic normal curve or the smooth points of a rational acnodal curve.

In the first case, by Lemma~\ref{lem:hyperplane}, since $n \ge C d^3 2^d$, the minimum number of ordinary hyperplanes is attained when all but one point is contained in a hyperplane and we get exactly $\binom{n-1}{d-1}$ ordinary hyperplanes.

In the second case, by Lemma~\ref{lem:coset2}, again since $n \ge C d^3 2^d$, the minimum number of ordinary hyperplanes is attained by a coset of an elliptic normal curve or the smooth points of a rational acnodal curve.
Lemmas~\ref{lem:coset1} and \ref{lem:2111} then complete the proof.
Note that the second term in the error term of Lemma~\ref{lem:2111} is dominated by the first term because of the lower bound on $n$, and that the error term here is negative by Lemma~\ref{lem:coset1}.
\end{proof}

Note that if we want to find the exact minimum number of ordinary hyperplanes spanned by a set of $n$ points in $\RR\PP^d$, $d \ge 4$, not contained in a hyperplane and where every $d$ points span a hyperplane, we can continue with the calculation of $[2, 1, \dotsc, 1; c]$ in the proof of Lemma~\ref{lem:2111}.
As seen in the proof of Lemma~\ref{lem:coset1}, this depends on $\gcd(d+1, n)$.
We also have to minimise over different values of $c \in H$, and if $n \equiv 0 \pmod{4}$, consider both cases $H \cong \ZZ_n$ and $H \cong \ZZ_{n/2} \times \ZZ_2$.

For example, it can be shown that if $d=4$, the minimum number is
\[ 
\begin{cases}
%\frac16 n^3 - n^2 + \frac{11}{6}n - 5 = 
\binom{n-1}{3} - 4 & \text{if } n \equiv 0 \pmod{5},\\
%\frac16 n^3 - n^2 + \frac{11}{6}n - 1 = 
\binom{n-1}{3} & \text{otherwise},
\end{cases}
\]
if $d=5$, the minimum number is
\[ 
\begin{cases}
%\frac{1}{24}n^4 - \frac{5}{12}n^3 + \frac43 n^2 - 2n =
\binom{n-1}{4} - \frac{1}{8}n^2 + \frac{1}{12} n - 1 & \text{if } n \equiv 0 \pmod{6},\\
%\frac{1}{24}n^4 - \frac{5}{12}n^3 + \frac{35}{24}n^2 - \frac{25}{12}n + 1 =
\binom{n-1}{4} & \text{if } n \equiv 1, 5 \pmod{6},\\
%\frac{1}{24}n^4 - \frac{5}{12}n^3 + \frac43 n^2 - \frac43 n  =
\binom{n-1}{4} -\frac{1}{8}n^2 + \frac{3}{4}n - 1 & \text{if } n \equiv 2, 4 \pmod{6},\\
%\frac{1}{24}n^4 - \frac{5}{12}n^3 + \frac{35}{24}n^2 - \frac{11}{4}n + 3 =
\binom{n-1}{4} - \frac{2}{3}n + 2 & \text{if } n \equiv 3 \pmod{6},
\end{cases}
\]
and if $d=6$, the minimum number is
\[
\begin{cases}
\binom{n-1}{5} - 6 & \text{if } n \equiv 0 \pmod{7},\\
\binom{n-1}{5} & \text{otherwise.}
\end{cases}
\]

\begin{proof}[Proof of Theorem~\ref{thm:main3}]
We first show that there exist sets of $n$ points, with every $d$ points spanning a hyperplane, spanning at least $\frac{1}{d+1} \binom{n-1}{d} + O\left(2^{-d/2}\binom{n}{\lfloor \frac{d-1}{2} \rfloor}\right)$ $(d+1)$-point hyperplanes.
Let $\delta^*$ be an elliptic normal curve or the smooth points of a rational acnodal curve.
By Propositions~\ref{prop:elliptic_group} and~\ref{prop:rational_group}, the number of $(d+1)$-point hyperplanes spanned by a coset $H \oplus x$ of $\delta^*$ is
\[ \frac{1}{(d+1)!}[\!\underbrace{1, \dotsc, 1}_\text{$d+1$ times}\!; c] \]
for some $c \in \delta^*$.
Note that
\[ [\!\underbrace{1, \dotsc, 1}_\text{$d+1$ times}\!; c] = d!\binom{n}{d} - d[2, \!\!\underbrace{1, \dotsc, 1}_\text{$d-1$ times}\!\!; c], \]
so if we take $H \oplus x$ to be a coset minimising the number of ordinary hyperplanes, then by Theorem~\ref{thm:main2}, there are
\begin{align}
&\mathrel{\phantom{=}} \frac{1}{d+1} \left( \binom{n}{d} - \binom{n-1}{d-1} \right) + O\left(2^{-\frac{d}{2}}\binom{n}{\lfloor \frac{d-1}{2} \rfloor}\right) \notag \\
&= \frac{1}{d+1} \binom{n-1}{d} + O\left(2^{-\frac{d}{2}}\binom{n}{\lfloor \frac{d-1}{2} \rfloor}\right)\label{eqn:d+1}
\end{align}
$(d+1)$-point hyperplanes.

Next let $P$ be an arbitrary set of $n$ points in $\RR\PP^d$, $d \ge 4$, where every $d$ points span a hyperplane.
Suppose $P$ spans the maximum number of $(d+1)$-point hyperplanes.
Without loss of generality, we can thus assume $P$ spans at least $\frac{1}{d+1}\binom{n-1}{d} + O\left(2^{-d/2}\binom{n}{\lfloor \frac{d-1}{2} \rfloor}\right)$ $(d+1)$-point hyperplanes.

Let $m_i$ denote the number of $i$-point hyperplanes spanned by $P$.
Counting the number of unordered $d$-tuples, we get
\[ \binom{n}{d} = \sum_{i \ge d} \binom{i}{d} m_i \ge m_d + (d+1)m_{d+1}, \]
hence we have
\[ m_d \le \binom{n}{d} - \binom{n-1}{d} - O\left(d2^{-\frac{d}{2}}\binom{n}{\lfloor \frac{d-1}{2} \rfloor}\right) = O\left(\binom{n-1}{d-1}\right), \]
and we can apply Theorem~\ref{thm:main1}.

In the case where all but $O(d2^d)$ points of $P$ are contained in a hyperplane, it is easy to see that $P$ spans $O(d2^d\binom{n}{d-1})$ $(d+1)$-point planes, contradicting the assumption.

So all but $O(d2^d)$ points of $P$ are contained in a coset $H \oplus x$ of a subgroup $H$ of $\delta^*$.
Consider the identity
\[ (d+1)m_{d+1} = \binom{n}{d} - m_d - \sum_{i \ge d+2} \binom{i}{d} m_i. \]
By Theorem~\ref{thm:main2} and Lemma~\ref{lem:coset2}, we know that $m_d \ge \binom{n-1}{d-1} - O\left(d2^{-d/2}\binom{n}{\lfloor \frac{d-1}{2} \rfloor}\right)$ and any deviation of $P$ from the coset $H \oplus x$ adds at least $c\binom{n-1}{d-1}$ ordinary hyperplanes for some sufficiently small absolute constant $c>0$.
Since we also have
\begin{align*}
\sum_{i \ge d+2} \binom{i}{d} m_i &= \binom{n}{d} - m_d - (d+1)m_{d+1}\\ 
&= \binom{n}{d} - \binom{n-1}{d-1} - \binom{n-1}{d} + O\left(d2^{-\frac{d}{2}}\binom{n}{\lfloor \frac{d-1}{2} \rfloor}\right)\\
&= O\left(d2^{-\frac{d}{2}}\binom{n}{\lfloor \frac{d-1}{2} \rfloor}\right),
\end{align*}
we can conclude that $m_{d+1}$ is maximised when $P$ is exactly a coset of a subgroup of $\delta^*$, in which case \eqref{eqn:d+1} completes the proof.
\end{proof}

Knowing the exact minimum number of ordinary hyperplanes spanned by a set of $n$ points in $\RR\PP^d$, $d \ge 4$, not contained in a hyperplane and where every $d$ points span a hyperplane then also gives the exact maximum number of $(d+1)$-point hyperplanes.

Continuing the above examples, for $d = 4$, the maximum number is
\[ 
\begin{cases}
%\frac{1}{120}n^4 - \frac{1}{12}n^3 + \frac{7}{24}n^2 - \frac{5}{12}n + 1 = 
\frac15\binom{n-1}{4} +\frac45 & \text{if } n \equiv 0 \pmod{5},\\
%\frac{1}{120}n^4 - \frac{1}{12}n^3 + \frac{7}{24}n^2 - \frac{5}{12}n + \frac15 = 
\frac15\binom{n-1}{4} & \text{otherwise},
\end{cases}
\]
for $d=5$, the maximum number is
\[ 
\begin{cases}
%\frac{1}{720}n^5 - \frac{1}{48}n^4 + \frac{17}{144}n^3 - \frac{7}{24}n^2 + \frac{11}{30}n =
\frac{1}{6}\binom{n-1}{5} + \frac{1}{48}n^2 - \frac{1}{72}n + \frac16 & \text{if } n \equiv 0 \pmod{6},\\
%\frac{1}{720}n^5 - \frac{1}{48}n^4 + \frac{17}{144}n^3 - \frac{5}{16}n^2 + \frac{137}{360}n - \frac16 =
 \frac{1}{6}\binom{n-1}{5} & \text{if } n \equiv 1, 5 \pmod{6},\\
%\frac{1}{720}n^5 - \frac{1}{48}n^4 + \frac{17}{144}n^3 - \frac{7}{24}n^2 + \frac{23}{90}n =
 \frac{1}{6}\binom{n-1}{5} + \frac{1}{48}n^2 - \frac{1}{8}n +\frac16 & \text{if } n \equiv 2, 4 \pmod{6},\\
%\frac{1}{720}n^5 - \frac{1}{48}n^4 + \frac{17}{144}n^3 - \frac{5}{16}n^2 + \frac{59}{120}n - \frac12  =
 \frac{1}{6}\binom{n-1}{5} +\frac{1}{9}n -\frac{1}{3} & \text{if } n \equiv 3 \pmod{6},
\end{cases}
\]
and for $d=6$, the maximum number is
\[ 
\begin{cases}
\frac17\binom{n-1}{6} +\frac67 & \text{if } n \equiv 0 \pmod{7},\\
\frac17\binom{n-1}{6} & \text{otherwise}.
\end{cases}
\]

%%% AUTHOR: optional appendix here
%\appendix %% you may comment this out if no Appendix
%\section*{Appendix}
%\section{}
%Material is placed here as needed.

%%% AUTHOR: optional acknowledgments here
\section*{Acknowledgments} 
We thank Peter Allen, Alex Fink, Misha Rudnev, and an anonymous referee for helpful remarks and for pointing out errors in a previous version.

%%% AUTHOR:
%%% Bibliography goes here. Note that the arXiv cannot process bibtex
%%% or biber bibliographies.  Example of acceptable bibliograpy format:
%\bibliographystyle{amsplain}
%\begin{thebibliography}{99}
%\end{thebibliography}

%% AUTHOR: You can generate such a bibliography from a .bib file by 
%% running pdflatex/bibtex/pdflatex/pdflatex and then pasting the .bbl file
%% between \begin{thebibliography} and \end{bibliography}

\newpage

%%% AUTHOR: Include a short description of each author following the
%%% structure below. Use the same short tags used previously.  
%%% Use \imageat{} and \imagedot{} instead of "@" and "." in
%%% email addresses-this replaces the symbols with graphics to avoid 
%%% e-mail address harvesting from the .pdf file
\begin{dajauthors}
\begin{authorinfo}[AL]
  Aaron Lin\\
  Department of Mathematics\\
  London School of Economics and Political Science\\
  United Kingdom\\
  aaronlinhk\imageat{}gmail\imagedot{}com
\end{authorinfo}
\begin{authorinfo}[KS]
  Konrad Swanepoel\\
  Department of Mathematics\\
  London School of Economics and Political Science\\
  United Kingdom\\
  k\imagedot{}swanepoel\imageat{}lse\imagedot{}ac\imagedot{}uk \\
  \url{http://personal.lse.ac.uk/swanepoe/}
\end{authorinfo}
\end{dajauthors}

\end{document}